\documentclass[final,leqno]{siamltex}
\pdfoutput=1
\usepackage{amsmath, amssymb}
\usepackage{color}
\usepackage{latexsym}
\usepackage{indentfirst}
\usepackage{graphicx}
\usepackage{placeins}
\usepackage{booktabs}
\usepackage{algorithm}
\usepackage{algpseudocode}
\usepackage{multirow}
\usepackage{subfig}
\usepackage{hhline}

\newtheorem{remark}[theorem]{Remark}

\newcommand{\ie}{\textit{i.e.}{}}
\newcommand{\eg}{\textit{e.g.}{}}
\newcommand{\egc}{\textit{e.g.,}{}}

\newcommand{\ud}{\,\mathrm{d}}

\newcommand{\Or}{\mathcal{O}}

\newcommand{\wt}[1]{\widetilde{#1}}

\newcommand{\abs}[1]{\left\lvert#1\right\rvert}

\newcommand{\barint}{\kern4pt \raise3.4pt\hbox{\vrule height.6pt
    width7pt} \kern-11pt \int}

\newcommand{\bvec}[1]{\mathbf{#1}}
\newcommand{\Vion}{V_{\mathrm{ion}}}
\newcommand{\Vxc}{\hat{V}_{\mathrm{xc}}}

\newcommand{\vr}{\bvec{r}}

\newcommand{\CS}{\mathcal{C}}
\newcommand{\IS}{\mathcal{I}}

\newcommand{\RS}{\mathcal{R}}
\newcommand{\JS}{\mathcal{J}} 
\newcommand{\LS}{\mathcal{L}}

\newcommand{\setN}{\left\{1,\ldots,N\right\}}
\newcommand{\setne}{\left\{1,\ldots,n_e\right\}}

\newcommand{\CSt}{\wt{\CS}}

\newcommand{\rr}[1]{{#1}}

\title{\rr{Computing localized representations of the Kohn-Sham subspace via randomization and refinement}}

\author{
	Anil Damle\thanks{Department of Mathematics, University of California, Berkeley, Berkeley, CA 94720 Email: \texttt{damle@berkeley.edu}}\and
	Lin Lin\thanks{Department of Mathematics, University of California, Berkeley, Berkeley, CA 94720 and Computational Research Division, Lawrence Berkeley National Laboratory, Berkeley, CA 94720. Email: \texttt{linlin@math.berkeley.edu}}\and
	Lexing Ying\thanks{Department of Mathematics and Institute for Computational and Mathematical Engineering, Stanford University, Stanford, CA 94305. Email: \texttt{lexing@math.stanford.edu}}
}

\begin{document}

\maketitle

\begin{abstract}
  Localized representation of the Kohn-Sham subspace plays an important
  role in quantum chemistry and materials science.  The recently
  developed selected columns of the density matrix (SCDM) method [J.
  Chem. Theory Comput. 11, 1463, 2015] is a simple and robust procedure for
  finding a localized representation of a set of Kohn-Sham orbitals from
  an insulating system. The SCDM method allows the direct construction
  of a well conditioned (or even orthonormal) and localized basis for the
  Kohn-Sham subspace. The SCDM algorithm avoids the use of an  optimization
  procedure and does not depend on any adjustable parameters.   The most
  computationally expensive step of the SCDM method is a 
  column pivoted QR factorization that identifies the important columns
  for constructing the localized basis set.  In this paper, we develop a two
  stage approximate column selection strategy to find the important
  columns at much lower computational cost. We demonstrate the
  effectiveness of this process using the
  dissociation process of a BH$_{3}$NH$_{3}$ molecule, an alkane chain
  and a supercell with $256$ water molecules. Numerical results for the large collection of water molecules show that two stage localization
  procedure can be more than $30$ times faster than the original SCDM algorithm and compares favorably with the popular Wannier90 package.
\end{abstract}



\section{Introduction}
\label{sec:intro}

Kohn-Sham density functional theory
(DFT)~\cite{HohenbergKohn1964,KohnSham1965} is the most widely used
electronic structure theory for molecules and systems in condensed
phase. The Kohn-Sham orbitals (a.k.a. Kohn-Sham wavefunctions) are
eigenfunctions of the Kohn-Sham Hamiltonian and are generally
delocalized, \ie~each orbital has significant magnitude across the
entire computational domain. Consequently, the information about atomic
structure and chemical bonding, which is often localized in real space,
may be difficult to interpret directly from Kohn-Sham orbitals. 
The connection between the delocalized orbitals, and localized ones can
be established
through a \textit{localization} procedure, which has been
realized by various numerical methods in the
literature~\cite{FosterBoys1960,MarzariVanderbilt1997,WannierReview,Gygi2009,SCDM,ELiLu2010,OzolinsLaiCaflischEtAl2013,AquilantePedersenMerasEtAl2006}.
The common goal of these methods is to find a set of
orbitals that are localized in real space and span the Kohn-Sham
invariant subspace, defined as the subspace as that spanned by the Kohn-Sham orbitals. \rr{For simplicity, we restrict our discussion to isolated systems and omit the discussion of Brillouin zone sampling in this manuscript.}

Mathematically, the problem of finding a localized representation of the
Kohn-Sham invariant subspace can be formulated as follows. Assume
the collection of Kohn-Sham orbitals are discretized in the real space representation as a tall
and skinny matrix $\Psi\in \mathbb{C}^{N\times n_{e}}$ with orthonormal columns and where
$N\gg n_{e}$. We seek to compute
a unitary transformation $Q\in\mathbb{C}^{n_{e}\times n_{e}}$ such that
the columns of $\Phi=\Psi Q$ are
\textit{localized}, \ie~each column of $\Phi$ becomes a sparse vector with spatially localized support
after truncating entries with relative magnitude smaller than a prescribed threshold. Here, $N$ is the number of grid
points in the discrete real space representation of each Kohn-Sham
orbital, and $n_{e}$ is the number of orbitals. In the absence of spin
degeneracy,  $n_{e}$ is also the number of electrons in the system. 

For a general matrix $\Psi$ it may not be possible to construct such a
$Q$ and obtain $\Phi$ with the desired structure. However, when
$\Psi$ represents a collection of  Kohn-Sham orbitals of an insulating
system, such localized orbitals \rr{generally} exist. \rr{An important exception are topological insulators with non-vanishing Chern numbers \cite{hasan2010colloquium,Brouder2007} \textemdash here we restrict our discussion to systems where localized functions are known to exist.} Their construction can be
justified physically by the ``nearsightedness'' principle for electronic
matter with a finite HOMO-LUMO gap~\cite{Kohn1996,ProdanKohn2005}. The
nearsightedness principle can be more rigorously stated as the single
particle density matrix being exponentially localized along the
off-diagonal direction in its real space
representation~\cite{BenziBoitoRazouk2013,Kohn1996,Blount,Cloizeaux1964a,Cloizeaux1964b,
Nenciu,LinLu2015}. 

The recently developed selected columns of the density matrix (SCDM)
procedure \cite{SCDM} provides a simple, accurate, and robust way of
constructing localized orbitals. Unlike many existing
methods~\cite{FosterBoys1960,MarzariVanderbilt1997,ELiLu2010,OzolinsLaiCaflischEtAl2013,MustafaCohCohenEtAl2015}, the SCDM method requires no initial
guess and does not involve a non-convex optimization procedure. \rr{The core ideas behind it also readily extend to systems with Brillouin zone sampling \cite{SCDMk}.} The SCDM procedure constructs localized orbitals directly from a
column selection procedure implicitly applied to the density matrix.
Hence, the locality of the basis is a direct consequence of the locality of the
density matrix. The SCDM method can be efficiently performed via a
single column pivoted QR (QRCP) factorization. Since efficient
implementation of the QRCP is available both in serial and
parallel computational environments through the LAPACK \cite{lapack} and
ScaLAPACK~\cite{Scalapack} libraries, respectively, the SCDM method can be
readily adopted by electronic structure software packages. 

From a numerical perspective, the computational cost of a QRCP factorization scales as
$\Or(N n_{e}^{2})$. The basic form of QRCP \cite{GVL} is not able to take full
advantage of level 3 BLAS operations. Hence for matrices of the same
size, the QRCP factorization can still be relatively expensive compared
to level 3 BLAS operations such as general matrix-matrix
multiplication (GEMM). The computational cost of the single QRCP is not necessarily an issue when the SCDM procedure is used as a post-processing tool, but is a
potential concern when the SCDM procedure needs to be performed repeatedly.  This, for instance, could occur in geometry optimization and molecular dynamics calculations with hybrid exchange-correlation
functionals~\cite{WuSelloniCar2009,GygiDuchemin2012}, where a localized
representation of the Kohn-Sham invariant subspace needs to be
constructed in each step to reduce the large computational cost
associated with the Fock exchange operator. In fact, \cite{SCDM} demonstrates how our existing SCDM algorithm may be used to accelerate Hartree-Fock exchange computations. Therefore, here we focus on accelerating the SCDM computation itself.

Practically, any QRCP algorithm may be used within the SCDM procedure.
In the serial setting, this includes recently developed methods based on
using random projections to accelerate and block up the column selection
procedure \cite{Martinsson,Gu}. In the massively parallel setting one
may alternatively use the recently developed communication avoiding
rank-revealing QR algorithm \cite{DemmelRRQR}.

In this paper, we demonstrate that the computational cost of the SCDM
procedure can be greatly reduced, to the extent that the column
selection procedure is no longer the dominating factor in the SCDM
calculation.
This is based on the observation that SCDM does not really require the
Q and R factors from the QRCP. In fact, only the pivots from
the QRCP are needed.
More specifically, we develop a two-stage column selection procedure that
approximates the behavior of the existing SCDM procedure at a much lower
computational cost. Asymptotically, the computational cost is only dominated
by two matrix-matrix multiplications of the form $\Psi Q$ to
construct the localized orbitals, at least one of which is needed in any localization
procedure starting from the input $\Psi$ matrix. Notably, the only adjustable parameters we introduce are an oversampling factor and truncation threshold. Both of which may be picked without knowledge of the problems physical structure.

The approximate column selection procedure consists of two stages. First,
we use a randomized procedure to select a set of candidate columns
that may be used in an SCDM style localization procedure. The number of
candidate columns is only $\Or(n_{e} \log n_{e})$ and is much smaller than
$N$. We may use these candidate columns to quickly construct a basis for
the subspace that is reasonably localized. In some cases, this fast
randomized procedure may provide sufficiently localized columns.
Otherwise, we
propose a subsequent refinement procedure to improve the quality of the
localized orbitals.  This is achieved by using a series of QRCP
factorizations for matrices of smaller sizes that may be performed concurrently.
Numerical results for physical systems obtained from the Quantum
ESPRESSO~\cite{QE} package indicate that the two-stage procedure yields
results that are nearly as good as using the columns selected by a full
QRCP based SCDM procedure. For large systems, the computational time is
reduced by more than one order of magnitude.

The remainder of this paper is organized as follows. In Section
\ref{sec:prelim} we present both a brief introduction to Kohn-Sham DFT
and a summary of the existing SCDM algorithm. Section \ref{sec:algo}
discusses the new two stage algorithm we propose, and details both the
randomized approximate localization stage and the refinement of the
column selection. Finally, Section \ref{sec:numer} demonstrates the
effectiveness of the algorithm for various molecules.

\section{Preliminaries}
\label{sec:prelim}

For completeness we first provide a brief introduction to Kohn-Sham density
functional theory, and the SCDM procedure for finding a localized
basis for the Kohn-Sham subspace.

\subsection{Kohn-Sham density functional theory}

For a given atomic configuration with $M$ atoms at locations $\{R_{I}\}_{I=1}^{M}$, 
KSDFT solves the nonlinear eigenvalue problem
\begin{equation}
  \begin{split}
    &\hat{H}[\hat{\rho};\{R_{I}\}]\hat{\psi}_{i} = \varepsilon_{i} \hat{\psi}_{i},\\
    &\hat{\rho}(\vr) = \sum_{i=1}^{n_e} \abs{\hat{\psi}_{i}(\vr)}^2, \quad \int
    \hat{\psi}^{*}_{i}(\vr) \hat{\psi}_{j}(\vr) \ud \vr = \delta_{ij}.
  \end{split}
  \label{eqn:KS}
\end{equation}
For simplicity we omit the spin degeneracy. 
The number of electrons is $n_{e}$ and the eigenvalues
$\{\varepsilon_{i}\}$ are ordered non-decreasingly.  The lowest
$n_e$ eigenvalues $\{\varepsilon_{i}\}_{i=1}^{n_e}$ are called the occupied state energies, and 
$\{\varepsilon_{i}\}_{j>n_e}$ are called the unoccupied state
energies.  We assume 
$\varepsilon_{g}:=\varepsilon_{n_e+1}-\varepsilon_{n_e}>0$. Here
$\varepsilon_{n_e}$ is often called the highest occupied molecular
orbital (HOMO), $\varepsilon_{n_e+1}$ the lowest unoccupied
molecular orbital (LUMO), and hence
$\varepsilon_{g}$ the HOMO-LUMO gap.  \rr{For extended systems, if $\varepsilon_{g}$ is uniformly bounded away from zero as the system size increases,}
the quantum system is an insulating system~\cite{Martin2004}.  
The eigenfunctions $\{\hat{\psi}_{i}\}_{i=1}^{n_e}$ define the
electron density $\hat{\rho}(\vr)$, which in turn defines the Kohn-Sham
Hamiltonian 
\begin{equation}
  \hat{H}[\hat{\rho};\{R_{I}\}] = -\frac12 \Delta +
  \hat{V}_{c}[\hat{\rho}] + \Vxc[\hat{\rho}] + \Vion[\{R_{I}\}].
\label{eqn:ksdft}
\end{equation}
Here $\Delta$ is the Laplacian operator for the kinetic
energy of electrons, 
\begin{equation*}
  \hat{V}_{c}[\hat{\rho}](\vr) \equiv
  \int \frac{\hat\rho(\vr')}{\abs{\vr-\vr'}} \ud \vr'
  \label{}
\end{equation*}
is the Coulomb potential, and $\hat{V}_{c}$ depends linearly with
respect to the electron density $\hat{\rho}$.  $\Vxc[\hat{\rho}]$
depends nonlinearly with respect to $\hat{\rho}$, and characterizes the
many body exchange and correlation effect. $\Vion[\{R_{I}\}]$ is an
external potential depending explicitly on the ionic positions, and
describes the electron-ion interaction
potential and is independent of $\hat{\rho}$.  Because the eigenvalue
problem (\ref{eqn:KS}) is nonlinear, it is often solved iteratively by a
class of algorithms called self-consistent field iterations
(SCF)~\cite{Martin2004}, until~\eqref{eqn:KS} reaches
self-consistency.

In a finite dimensional discretization of Eq.~\eqref{eqn:ksdft}, let $N$
be the number of degrees of freedom. Using a large basis set such as the
plane-wave basis set, we have $N=c n_{e}$ and $c$ is a large constant
that is often $10^{2}\sim 10^{4}$. Due to this large constant, we
explicitly distinguish $N$ and $n_{e}$ in the complexity analysis below.
For large
scale systems, the cost for storing the Kohn-Sham orbitals is $\Or(N
n_e)$, and the cost for computing them is generally $\Or(N n_e^2)$
and scales cubically with respect to $n_e$.  In modern KSDFT
calculations the Hartree-Fock exact exchange term is also often taken
into account in the form of hybrid
functionals~\cite{PerdewErnzerhofBurke1996,Becke1993}.  The
computational cost for this step not only scales as $\Or(Nn_e^{2})$ but
also has a large pre-constant.

When the self-consistent solution of the Kohn-Sham equation is obtained,
the existence of finite HOMO-LUMO gap has important implications on the
collective behavior of the occupied Kohn-Sham orbitals
$\{\hat{\psi}_{i}\}_{i=1}^{n_e}$. Since any non-degenerate linear
transformation of the set of Kohn-Sham orbitals yields exactly the same
physical properties of a system, the physically relevant
quantity is the subspace spanned by the Kohn-Sham orbitals
$\{\hat{\psi}_{i}\}_{i=1}^{n_e}$. Various
efforts~\cite{FosterBoys1960,MarzariVanderbilt1997,WannierReview,Gygi2009,ELiLu2010,OzolinsLaiCaflischEtAl2013}
have been made to utilize this fact and to find a set of localized
orbitals that form a compressed representation of a Kohn-Sham subspace.
In other words, we find a set of functions
$\{\hat{\varphi}_{i}\}_{i=1}^{n_e}$ whose span is the same as the
span of $\{\hat{\psi}_{i}\}_{i=1}^{n_e}$.
Compared to each Kohn-Sham orbital $\hat{\psi}_{i}$ which is delocalized in
the real space, each compressed orbital $\hat{\phi}_{i}$ is often localized
around an atom or a chemical bond. Hence working with $\hat{\phi}_{i}$'s
can reduce both the storage and the computational cost.

Assume we have access to $\hat{\psi}_{j}(\vr)$'s evaluated at a set of
discrete grid
points $\{\vr_{i}\}_{i=1}^{N}$.  Let $\{\omega_{i}\}_{i=1}^{N}$ be a set
of positive integration weights associated with the grid points
$\{\vr_{i}\}_{i=1}^{N}$, then the discrete orthonormality condition
is given by
\begin{equation}
  \sum_{i=1}^{N} \hat{\psi}_{j}(\vr_{i}) \hat{\psi}_{j'}(\vr_{i}) \omega_{i} =
  \delta_{jj'}.
  \label{eqn:orthonormal_discrete}
\end{equation}
Let $\hat{\psi}_{j}=[\hat{\psi}_{j}(\vr_{1}), \hat{\psi}_{j}(\vr_{2}), \ldots,
\hat{\psi}_{j}(\vr_{N})]^{T}$ be a column vector, and $\hat{\Psi}=[\hat{\psi}_{1},  \ldots,  \hat{\psi}_{n_e}]$
be a matrix of size $N\times n_e$. We call $\hat{\Psi}$ the
\textit{real space representation} of the Kohn-Sham orbitals and define diagonal matrix
$W=\mathrm{diag}[\omega_{1},\ldots,\omega_{N}]$.

\rr{Our method requires the Kohn-Sham orbitals to be be represented on a set of real space grid points. This is the case for a plane-wave basis set, as well as other representations such as finite differences, finite elements and wavelets. For instance, if
the Kohn-Sham orbitals are represented using the plane-wave basis
functions, their real space representation can be obtained on a uniform
grid efficiently with the fast Fourier transform (FFT) technique and in this case $\omega_{i}$ takes the same constant value for all $i$. It is in this setting that our method is of particular interest. However, this procedure is also applicable to other basis sets such as Gaussian type orbitals or numerical atomic orbitals when a real space representation of the basis functions is readily available. Therefore, our method is amenable to most electronic structure software packages.}




We define $\Psi=W^{\frac12} \hat{\Psi}$ such that the discrete orthonormality
condition in Eq.~\eqref{eqn:orthonormal_discrete} becomes
$\Psi^{*}\Psi=I$, where $I$ is an identity matrix of size $n_e$.
We now seek a compressed basis for the span of $\Psi$, denoted by the
set of vectors $\Phi=[\phi_{1}, \ldots,  \phi_{n_e}]$ where each
$\phi_{i}$ is a sparse vector with spatially localized support after truncating entries with small
magnitudes. In such case, $\phi_{i}$ is called a localized vector. 

\subsection{Selected columns of the density matrix}

As opposed to widely-used procedures such as MLWFs \cite{WannierReview}, the key difference in
the SCDM procedure is that the localized orbitals $\phi_{i}$ are
obtained directly from columns of the density matrix $P = \Psi\Psi^*$.
The aforementioned nearsightedness principle states that, for
insulating systems, each column of the matrix $P$ is localized. As a result, selecting any linearly independent subset of $n_e$ of them
will yield a localized basis for the span of $\Psi.$ However, picking
$n_e$ random columns of $P$ may result in a poorly conditioned basis if,
for example, there is too much overlap between the selected
columns.  Therefore, we would like a means for choosing a well
conditioned set of columns, denoted $\CS = \left\{c_1,c_2,\ldots,c_{n_e} \right\},$ to use as the localized
basis. Intuitively we expect such a basis to select columns to minimize
overlaps with each other when possible.


\rr{This is algorithmically accomplished with a QRCP factorization (see, \egc~\cite{GVL}). More specifically, given a matrix $A$ a QRCP seeks to compute a permutation matrix $\Pi$ such that the leading sub-matrices $\left(A\Pi\right)_{1,\ldots,k,:}$ for any applicable $k$ are as well conditioned as possible. In particular, if we let $\CS$ denote the columns selected by the first $n_e$ columns of $\Pi$ then $A_{:,\CS}$ should be a well conditioned set of $n_e$ columns of $A.$}

\rr{In our setting, this means we would ideally compute an QRCP factorization of the matrix $P$ to identify $n_e$ well conditioned columns from which we may construct a localized basis. However, this 
would be highly costly since $P$ is a large matrix of size 
$N$. Fortunately, we may equivalently compute the
set $\CS$ by computing a QRCP factorization of $\Psi^*$ \rr{or, in fact, any matrix $U$ with orthogonal columns such that $P=UU^*.$} More
specifically, we compute
\begin{equation}
\label{eqn:qrcp}
\Psi^*\Pi = Q\begin{bmatrix} R_1 & R_2 \end{bmatrix},
\end{equation}
and the first $n_e$ columns of $\Pi$ encode $\CS.$}

The SCDM procedure to construct an orthonormal set of localized basis elements, denoted $\phi_i$ for $i=1,\ldots,n_e$, and collected as columns of the matrix $\Phi$ is presented in its simplest form in Algorithm \ref{alg:scdm}.
\begin{algorithm}
\caption{The SCDM algorithm}
\label{alg:scdm}
\begin{algorithmic}[1]
\Statex Given: the Kohn-Sham orbitals $\Psi$
\State Compute a column pivoted QR of $\Psi^*$, $\Psi^*\Pi = Q\begin{bmatrix} R_1 & R_2 \end{bmatrix}$
\State Compute $\Phi = \Psi Q$ or, alternatively, $\Phi^* = \begin{bmatrix} R_1 & R_2 \end{bmatrix}\Pi^*$
\Statex Output: a localized basis for the Kohn-Sham subspace $\Phi$
\end{algorithmic}
\end{algorithm} 
In such form the algorithm requires knowledge of the orthogonal factor
from the QRCP. However, an alternative description simply requires the
column selection $\CS.$ We may equivalently write the SCDM algorithm as
in Algorithm \ref{alg:scdm_no_q}. Note that in Algorithm
\ref{alg:scdm_no_q}, the cost of the QR factorization for the
matrix $\left(\Psi_{\CS,:}\right)^*$ is only $\Or(n_{e}^{3})$.
\begin{algorithm}
\caption{An alternative version of the SCDM algorithm}
\label{alg:scdm_no_q}
\begin{algorithmic}[1]
\Statex Given: the Kohn-Sham orbitals $\Psi$
\State Compute $\CS$ associated with a column pivoted QR of $\Psi^*$
\State Compute the QR factorization $\left(\Psi_{\CS,:}\right)^* = QR$
\State Compute $\Phi = \Psi Q$
\Statex Output: a localized basis for the Kohn-Sham subspace $\Phi$
\end{algorithmic}
\end{algorithm}

\begin{remark}
There are various equivalent ways to construct the SCDM
algorithm. While the simple presentation here differs slightly
from the original presentation \cite{SCDM} \rr{the two are mathematically equivalent up to a choice of sign for the columns of $\Phi$. The original presentation corresponds to the computation of a QR factorization of $\left(\Psi_{\CS,:}\right)^*$ via a Cholesky factorization of $\left(\Psi_{\CS,:}\right)\left(\Psi_{\CS,:}\right)^*.$ QR factorizations are not
unique, there is always ambiguity up to diagonal matrix with
entries on the unit circle. However, such an ambiguity does not have any
affect on the localization.} 
\end{remark}

\rr{This second interpretation allows us to briefly explain why this algorithm constructs localized orbitals. Let $D$ be a diagonal matrix with $\pm 1$ on the diagonal such that $DR$ has positive diagonal entries. This means that
\[
Q = \left(\Psi_{\CS,:}\right)^*R^{-1}D
\] 
where $R^{-1}$ is a Cholesky factor of $\left[\left(\Psi_{\CS,:}\right)\left(\Psi_{\CS,:}\right)^*\right]^{-1}.$ Importantly, $\left(\Psi_{\CS,:}\right)\left(\Psi_{\CS,:}\right)^* = P_{\CS,\CS}$ and, therefore, exhibits off diagonal decay so long as $P_{\CS,\CS}$ is well conditioned. This property is then inherited by $R^{-1}$ \cite{BenziBoitoRazouk2013}. Finally, since $P_{\CS,:} = \Psi\left(\Psi_{\CS,:}\right)^*$ 
\[
\Phi = P_{\CS,:}R^{-1}D
\]
we may conclude that it is well localized \textemdash $P_{\CS,:}$ is well localized and $R^{-1}D$ does not destroy that locality. Importantly, here we see that all the factorizations we are performing can be thought of as involving sub-matrices of $P$.}

The overall computational cost of the algorithm is $\Or(N n_e^2),$ and
practically the cost is dominated by the single QRCP factorization regardless of the version used. Another key feature of the
algorithm, especially for our modifications later, is that because we
are effectively working with the spectral projector $P,$ the method
performs equivalently if a different orthonormal basis for the range of
$\Psi$ is used as input. In physics terminology, the SCDM procedure is
gauge-invariant. Lastly, the key
factor in forming a localized basis is the selection of a well
conditioned subset of columns. Small changes to the selected columns, provided they remain nearly as well conditioned, may
not significantly impact the overall quality of the basis.

\section{The approximate column selection algorithm}
\label{sec:algo}

When the SCDM procedure is used as a post-processing tool for a single
atomic configuration, the computational cost is usually affordable. In
fact in such a situation, the most time consuming part of the
computation is often the I/O related to the $\Psi$ matrices especially
for systems of large sizes.  However, when localized orbitals need to be
calculated repeatedly inside an electronic structure software package,
such as in the context of hybrid functional calculations with
geometry optimization or \textit{ab initio} molecular dynamics
simulations, the computational cost of SCDM can become relatively large.
Here we present an algorithm that significantly accelerates the SCDM
procedure.

The core aspect of the SCDM procedure is the column selection procedure. Given a set of appropriate columns the requisite orthogonal transform to construct the SCDM may be
computed from the corresponding rows of $\Psi$, as seen in Algorithm~\ref{alg:scdm_no_q}. Here we present a two stage procedure for accelerating this selection of columns and hence the computation
of $\Phi$. First, we construct a set of approximately localized orbitals
that span the range of $\Psi$ via a randomized method that requires only
$\Psi$ and the electron density $\rho,$ though if $\rho$ is not given it may be computed directly from $\Psi$ without increasing the asymptotic computational complexity.  We then use this approximately localized basis as the
input for a procedure that refines the selection of columns from which the localized basis is ultimately constructed. This is done by using the approximate locality to carefully partition the column selection
process into a number of small, local, QRCP factorizations.
Each small QRCP may be done in
parallel, and operates on matrices of much smaller dimension than
$\Psi.$

\subsection{Approximate localization}

The original SCDM procedure, through the QRCP, examines all $N$ columns
of $\Psi^*$ to decide which columns to use to construct $Q$. However,
physical intuition suggests that it is often not necessary to visit all
columns to find good pivots. For instance, for a molecular system in
vacuum, it is highly unlikely that a pivot comes from a column of the
density matrix corresponding to the vacuum space away from the molecule.
This inspires us to accelerate the column selection procedure by
restricting the candidate columns.

This is accomplished by generating $\Or\left(n_e\log n_e\right)$ 
independent and identically distributed (i.i.d.) random sample columns,
using the normalized electron density as the probability distribution
function (pdf).  Indeed, if a column of the density matrix corresponds
to the vacuum, then the electron density is very small and hence the
probability of picking the column is very low. In statistics this
corresponds to leverage score sampling, see, \egc~\cite{MahoneyDrineas}. This randomized version of the SCDM
algorithm is outlined in Algorithm~\ref{alg:rand}.


\begin{algorithm}
\caption{Computing an approximately localized collection of basis vectors}
\label{alg:rand}
\begin{algorithmic}[1]
\Statex Given: Kohn-Sham orbitals $\Psi,$ electron density $\rho,$ concentration $\gamma,$ and failure probability $\delta$
\State Sample $(n_e / \gamma) \log n_e / \delta$ elements from $\setN$ based on the discrete distribution
\[
\Pr\left(\left\{j\right\}\right) = \rho(j)/n_e
\]
and denote this set $\CSt$
\State Compute the column pivoted QR factorization $$\left(\Psi_{\CSt,:}\right)^*\Pi = QR$$
\State Form approximately localized basis $\wt{\Phi} = \Psi Q$
\end{algorithmic}
\end{algorithm} 

To complete our discussion of Algorithm \ref{alg:rand} we must justify the sub-sampling procedure used to select $\CSt$. In order to do so we introduce a simple model for the column selection procedure based on the idea that columns ``similar'' to the ones selected by Algorithm~\ref{alg:scdm} will work well to compute an approximately localized basis. Our underlying assumption is that a set of columns will serve to approximately localize the basis if it contains at least one column in each region where one of the $\phi_i$ is large. Because the $\phi_i$ constructed via the SCDM procedure decay exponentially, this is analogous to saying that any column associated with a grid point close enough to the ``true'' grid point used will suffice. Though, by avoiding explicit use of spatial relations of grid points our algorithm and its parameters are not dependent on the physical geometry.

To codify this postulate, we let $\IS_i \subset \setN$ be the smallest non-empty set such that 
\begin{equation}
\sum_{j\in\IS_i} \lvert\phi_i(j)\rvert^2 \geq \gamma,
\label{eqn:gamma}
\end{equation}
where $\gamma \in (0,1)$. If multiple such sets exist we select the one
that maximizes ${\sum_{j\in\IS_i} \lvert\phi_i(j)\rvert^2}$. Now, we may write our assumption more concretely: a column $c_i \in \setN$ suffices to approximately construct $\phi_i$ if it is contained in $\IS_i$. \rr{Taking $\gamma$ to be sufficiently small would enforce adequate sampling to ensure the columns selected by Algorithm~\ref{alg:scdm_no_q} are selected to be part of $\CSt.$ However, in practice this is not necessary for the construction of a localized basis and by choosing a larger $\gamma$ we allow for other columns near the peak of $\phi_i$ to act as good surrogates.} 


Under this assumption, to approximately localize the basis, we must
simply ensure that $\CSt$ contains at least one distinct column in each of the sets
$\IS_1,\IS_2,\ldots,\IS_{n_e},$ \ie~we need a one to one matching between sets and
columns. Theorem \ref{thm:sample} provides an upper bound on the
required cardinality of $\CSt$ to ensure it may be used to approximately
localize the basis with high probability. We do require an additional
mild assumption that ensures the sets $\IS_i$ do not simultaneously
overlap significantly and have small support. 
\begin{theorem}
\label{thm:sample}
Let $\eta$ be the largest constant such that there exist disjoint subsets $\IS^s_i \subseteq \IS_i, \; i=1,\ldots,n_e$ each satisfying 
\[
\sum_{j\in\IS^s_i} \lvert\phi_i(j)\rvert^2 \geq \eta\sum_{j\in\IS_i}
\lvert\phi_i(j)\rvert^2.
\]
The set $\CSt$ constructed by sampling $$m \geq  \left(n_e / \eta\gamma \right) \log n_e / \delta$$ elements with replacement from $\setN$ based on the discrete distribution 
\[
\Pr\left(\left\{j\right\}\right) = \rho(j)/n_e
\]
contains an unique element in each $\IS_i$ for $i=1,\ldots,n_e$ with probability $1-\delta.$
\end{theorem} 
\begin{proof}
Let $\mathcal{F}$ be the event that $\CSt$ does not contain a distinct element in one of the sets $\IS_i.$ We may write
\[
\Pr\left(\mathcal{F}\right) \leq \Pr\left(\left\{\CSt \cap \IS^s_i = \emptyset \text{ for some } i\right\}\right)
\]
because requiring $\CSt$ to contain an element in each $\IS^s_i$ implies that $\CSt$ contains a distinct element in each $\IS_i$. Subsequently, by a union bound
\[
\Pr\left(\mathcal{F}\right)\leq \sum_{i=1}^{n_e} \Pr\left(\left\{\CSt \cap \IS^s_i = \emptyset\right\}\right).
\]
The event $\left\{\CSt \cap \IS^s_i = \emptyset\right\}$ is simply the probability that none of the $m$ samples fall in $\IS^s_i.$ Because
\[
\rho(j) = \sum_{i=1}^{n_e} \lvert \phi_i(j)\rvert^2,
\]
we may find the lower bound for the probability of selecting an element $j$ in $\IS_i$ as
\[
\Pr\left(\left\{j \in \IS^s_i\right\}\right) \geq \eta\gamma/n_e.
\]
Consequently, 
\[
\sum_{i=1}^{n_e} \Pr\left(\left\{\CSt \cap \IS^s_i = \emptyset\right\}\right) \leq n_e (1-\eta\gamma/n_e)^m
\]
and to ensure the probability of missing any of the sets $\IS^s_i$ is less than $\delta$ we may simply enforce
\[
m \geq \frac{n_e}{\eta\gamma} \log n_e / \delta.
\]
\end{proof}

\begin{remark}
The parameter $\eta$ is
necessary in the proof, but
algorithmically we simply assume it to be one, alternatively, one could
consider choosing $1 / (\gamma \eta)$ rather than just $1 / \gamma$ as the
oversampling factor.
\end{remark}

If, for example, we take $\gamma = 1/2$ and $\eta = 1/2$ this bound says
$4n_e\log n_e/\delta$ samples suffices for the approximate localization
procedure. As expected, when either the failure probability $\delta$ or
the cardinality of $\IS_i$ go to one the required number of samples
grows. Furthermore, since $\gamma = \min_{i} \max_{j} \lvert \phi_i(j) \rvert^2$ corresponds to each of the sets $\IS_i$ containing a single point, that is a lower bound on how small $\gamma$ can be theoretically. We remark that this theoretical bound may be pessimistic for two reasons. One is its use of the union bound. The other is the introduction of $\eta$. The disjoint requirement simplifies the assignment of selected columns to sets for the proof, but is a relaxation of what is really needed.

\subsection{Accelerating the SCDM procedure using an approximately
localized basis}

Once the approximate localized orbitals $\wt{\Phi}$ are obtained, we
would like to perform a refinement procedure to further localize the
basis. We do this by taking advantage of the locality of the input to
the SCDM procedure. \rr{In Algorithm~\ref{alg:rand} we approximate the behavior of the SCDM algorithm by restricting the number of columns of $\Psi^*$ that the QRCP factorization can select from. However, once we have a somewhat localized basis we can efficiently take more columns of $\Psi^*$ into consideration. This allows us to better approximate the original SCDM algorithm and, therefore, construct a more localized basis. We accomplish this with a procedure that resembles the tournament pivoting strategy for computing a QRCP. We first compute a bunch of small QRCP factorizations, each involving columns associated with the support of a subset of the rows of $\wt{\Phi}^*$. This is computationally feasible because the rows are already somewhat localized. Lastly, because this procedure generates more candidate columns than needed we perform one final QRCP to try and get as well conditioned a set of columns of $\Psi^*$ as possible.}

Algorithmically, we first need to select a superset of the ultimately
desired columns from which to select the final columns used in the
localization. To construct such a superset we consider each approximately localized
orbital and how we may refine it. Each orbital, once approximately
localized, only exerts influence on, \ie~significantly overlaps with,
nearby localized orbitals. Hence, we may refine the selected columns locally. This
means that for each orbital we may simply figure out which orbitals
it substantially overlaps with and compute a QRCP on just those orbitals (rows) of
$\wt{\Phi}^*$ while simultaneously omitting columns with small norm over
those rows. This process will yield a small number of selected columns
that we add to a list of potential candidate columns for the final
localization. However, because we repeat this process for each localized
orbital we might have more than $n_e$ total candidate columns by a small multiplicative factor. Therefore, we perform one final column pivoted QR factorization on these candidate columns to select the final set $\CS$. 

In principle, while the localized orbitals get small on large portions of the domain, they are not actually zero. Hence, for any given $\epsilon$ we say two orbitals substantially overlap if there is any spatial point where they both have relative magnitude greater than $\epsilon$. We find in practice that choosing $\epsilon$ close to the relative value at which one considers an orbital to have vanished suffices, though taking $\epsilon \rightarrow 0$ one recovers the original SCDM algorithm. Importantly, this parameter is completely independent of the geometry of the problem under consideration and thus does not require detailed knowledge of the system to set.

We detail our complete algorithm for computing the orthogonalized SCDM from approximately localized input in Algorithm~\ref{alg:refine}. The only parameter is how small a column is for it to be ignored. At a high level, the goal is simply to generate additional candidate
columns by allowing each orbital to interact with its immediate
neighbors. Provided that each orbital is only close to a small number of
others, and that the approximate localization is sufficient this procedure may be performed quickly. 

\begin{algorithm}
\caption{Refining an approximately localized collection of basis vectors}
\label{alg:refine}
\begin{algorithmic}[1]
\Statex Given:  approximately localized Kohn-Sham orbitals $\wt{\Phi}$ and column tolerance $\epsilon$
\State $\JS_i = \left\{j\in \setN \mid \lvert \wt{\phi_i}\rvert > \epsilon \max_{k} \lvert \wt{\phi_i} (\vr_k) \rvert \right\}$ for $i = 1,\ldots,n_e$
\For{$i = 1,\dots, n_e$}
  \State Set $\RS_i = \left\{j \in \setne \mid \JS_i \cap \JS_j \neq \emptyset \right\}$ 
  \State Set $\displaystyle \LS_i = \bigcup_{j\in \RS_i} \JS_j$
  \State Compute a column pivoted QR factorization of $\left(\left.\rr{\wt{\Phi}}\right.^*\right)_{\RS_i,\LS_i}$ and denote the pivot columns $\CS_i$
\EndFor
\State Set $\CSt = \cup_i \CS_i$
\State Compute the column pivoted QR factorization $$\left(\rr{\wt{\Phi}}_{\CSt,:}\right)^*\Pi = QR$$
\State Form the localized basis $\Phi = \wt{\Phi} Q$
\end{algorithmic}
\end{algorithm} 

To illustrate the behavior of this algorithm, we sketch the behavior of Algorithm~\ref{alg:refine} in two cases. In one case, after the approximate localization, we have two sets of orbitals whose support sets after truncation are disjoint. This is shown in Figure~\ref{fig:disjoint}. Here, simply computing two independent QRCP factorizations is actually equivalent to computing the QRCP of the entire matrix. As we see Algorithm~\ref{alg:refine} partitions the orbitals into two sets and then only considers the columns with significant norm over the orbital set. In the second case, \rr{illustrated in Figure~\ref{fig:chain}}, we have a chain of orbitals whose support set after truncation forms a connected region in the spatial domain. In this situation we do not actually replicate the computation of a QRCP of the whole matrix, but rather for each orbital we compute a local QRCP ignoring interactions with distant orbitals. More specifically, any column mostly supported on a given orbital will be minimally affected by orthogonalization against columns associated with distant orbitals. Therefore, we may simply ignore those orthogonalization steps and still closely match the column selection procedure. 

\begin{figure}[ht!]
  \centering
  \includegraphics[width = 1\textwidth]{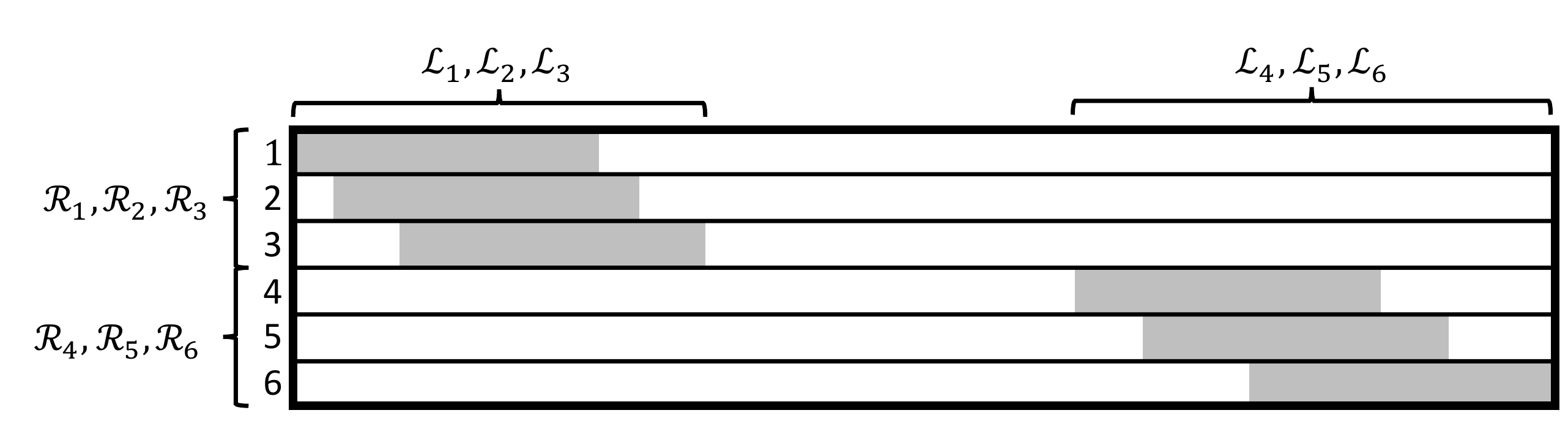}
  \includegraphics[width = .6\textwidth]{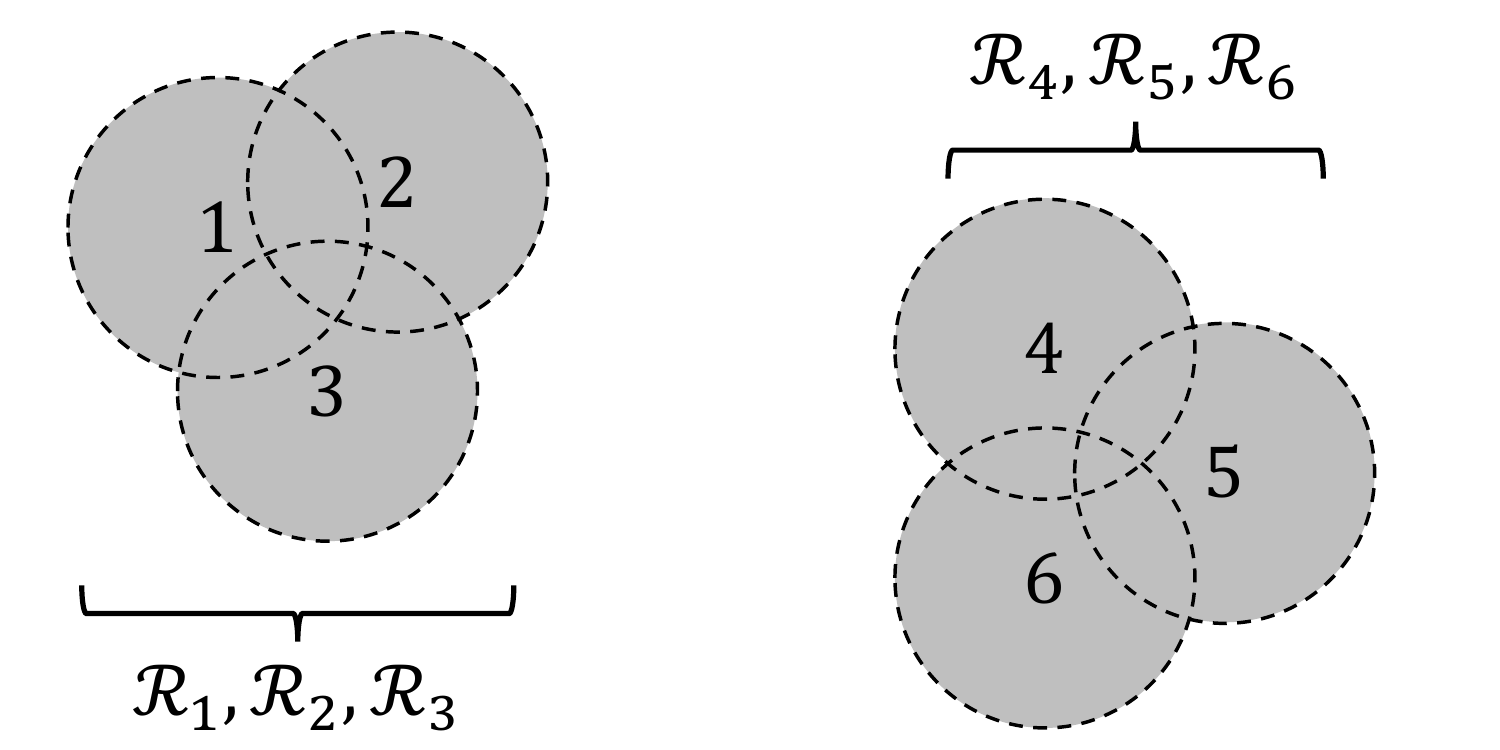}
  \caption{Matrix (top) and physical domain (bottom) associated with two collections of approximately localized orbitals whose support (lightly shaded region) is disjoint after truncation.}
  \label{fig:disjoint}
\end{figure}

\begin{figure}[ht!]
  \centering
  \includegraphics[width = 1\textwidth]{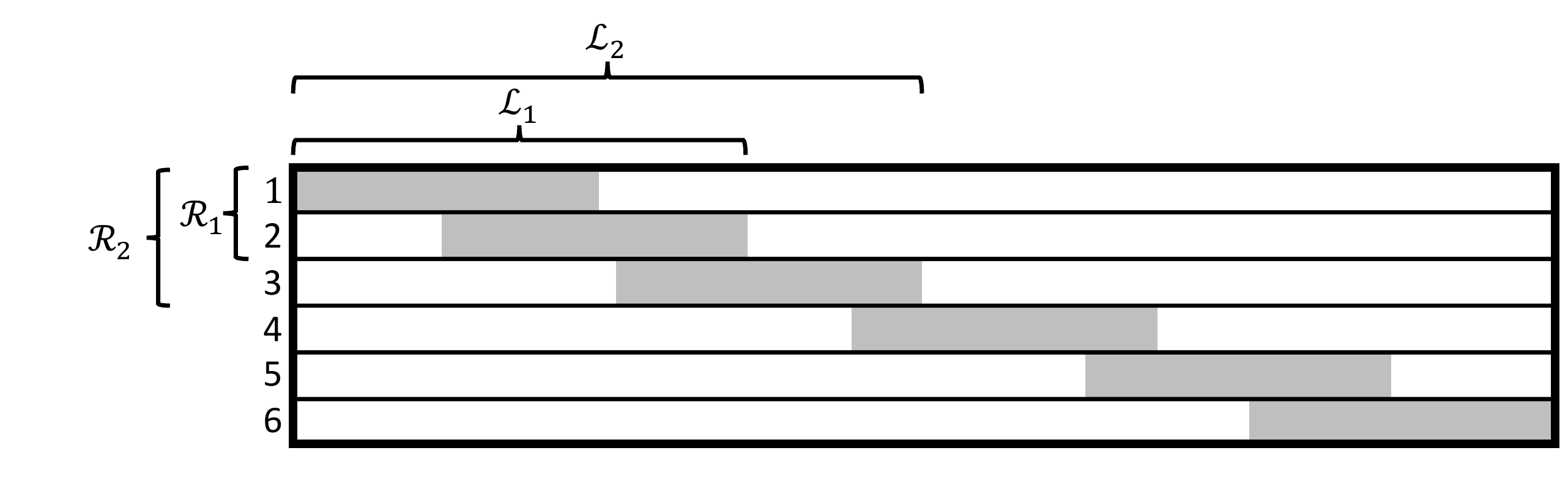}
  \includegraphics[width = .6\textwidth]{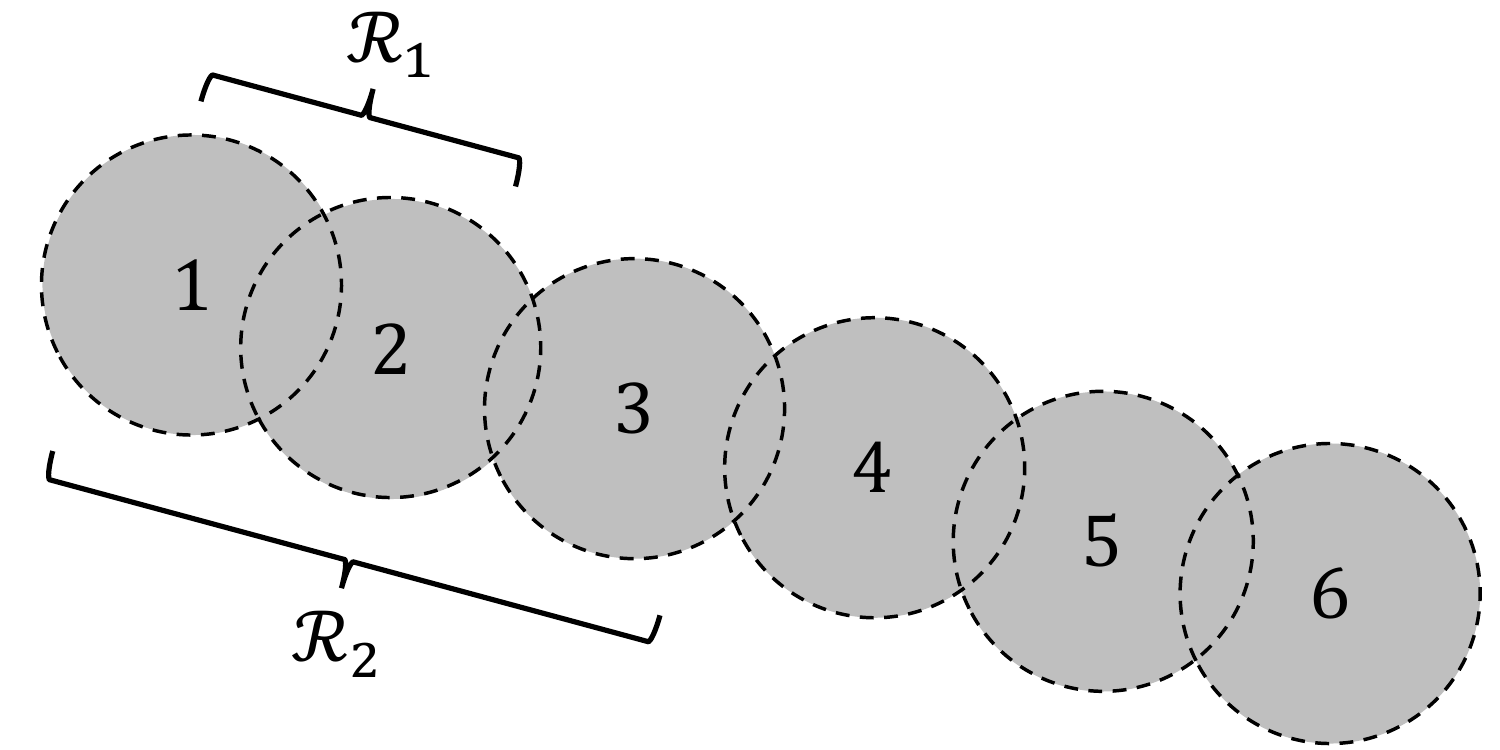}
  \caption{Matrix (top) and physical domain (bottom) associated with a chain of approximately localized orbitals whose support (lightly shaded region) is connected after truncation.}
  \label{fig:chain}
\end{figure}

\begin{remark}
In Algorithm~\ref{alg:refine} it may sometimes be the case that several of the $\RS_i$ are identical (\eg~as in Figure~\ref{fig:disjoint}). In this case, as a matter
of practical performance optimization, one may simply skip instances of
the loop over $i$ that would be computing a QRCP of the exact same
matrix as a prior instance of the loop. Similarly, if there is a set
$\mathcal{S}\subset \setne$ such that $\left(\cup_{i\in \mathcal{S}}
\RS_i \right) \cap \left(\cup_{i\in\mathcal{S}^c} \RS_i\right) = \emptyset$ and $\lvert \cup_{i\in \mathcal{S}} \RS_i \rvert$ is small, we may combine the instances of the loop corresponding to $i\in\mathcal{S}$ into a single small QRCP by simply using $\cup_{i\in \mathcal{S}} \RS_i$ as the set of rows. This avoids redundant work if a small collection of orbitals just interact amongst themselves and are disjoint from all others.   
\end{remark}

\subsection{Computational complexity}

The computational cost of Algorithms \ref{alg:rand} and \ref{alg:refine}
is formally limited by the cost of computing a single general matrix-matrix
multiplication (GEMM) $\Psi Q$ at cost $\Or(n_e^2 N).$ However, this single BLAS 3
operation must appear in any localization procedure starting with the
$\Psi$ matrix as an input.
This GEMM operation is also highly optimized both in sequential and parallel
computational environments. Furthermore, at large scale one can
approximate the product with $\Or(n_e N)$ complexity using sparse linear algebra, since the support of the result
is sparse and one may infer the support based on the column indices from
which $Q$ is built. For these reasons, we let $T_{mult}$ represent the
cost of computing $\Psi Q$. 

Using this notation, the computational cost of Algorithm~\ref{alg:rand}
is $$\Or(n_e N \log n_e + n_e^3 \log n_e) + T_{mult}.$$ More
specifically, the random selection of columns costs $\Or(N n_e \log
n_e)$ and the subsequent QR factorization costs $\Or(n_e^3 \log n_e).$
If we assume that the support of the approximately localized basis used
as input for Algorithm~\ref{alg:refine} and the number of nearby orbital
centers are bounded by some constant independent of $N$ and $n_e,$ which
is reasonable for models where the molecule is growing and the
discretization quality remains constant, the computational cost is
$$\Or(n_e N + n_e^3) +T_{mult}.$$ Under these assumptions each of the,
at most, $n_e$ small QR factorizations has constant cost. 
However, the procedure for finding $\JS_i$ introduces a dependency on $N$.
While the computational cost of the
refinement procedure and randomized algorithms are broadly similar, the
randomized algorithm is significantly faster in practice because drawing
the random samples is cheaper than the support computations needed in
the refinement algorithm.

\rr{Lastly, our algorithms are memory efficient. In fact, for practical purposes their memory cost is bounded by the cost of storing $\Psi$ plus a few work arrays of size $N.$ Besides the storage for $\Psi,$ all of the matrices we operate on cost at most $\Or(n_e^3 \log n_e)$ to store and may be discarded after use. Furthermore, the QR factorizations and matrix multiplication may be done in place excepting a few work arrays of length $N.$}

\section{Numerical examples}
\label{sec:numer}

To demonstrate the performance of our method we use three examples that
capture the different facets of our algorithm.  The first example is the
dissociation of a BH$_{3}$NH$_{3}$ molecule, and the second example is
the alkane chain. We select these two examples not because they are
computationally expensive, but that they clearly demonstrates that 
our approximate localization algorithm is effective in two very
different regimes.  In particular, the effectiveness of the algorithm is
independent of whether localized
orbitals form one single group or multiple disconnected groups. Our third example is a large supercell with $256$ water
molecules. We demonstrate the performance gains over the existing SCDM
method and provide a comparison with Wannier90 \cite{wannier90}.

In all of the examples here we assume we have access to the electron density
$\rho$ from the electronic structure calculation, and therefore 
exclude its computation from the timings of our randomized method. \rr{We use $3 n_e \log n_e$ samples in the randomized algorithm, corresponding to $\gamma = 1/3,$ for all of the experiments.} 
Furthermore, to more clearly illustrate the advantages of our method, we
separately report timings for computing the orthogonal
transform $Q$ that localizes the orbitals and subsequent computation of the localized orbitals by a single matrix product. Here
we only consider the orthogonalized SCDM, as discussed in this paper.
For the refinement algorithm we set the relative truncation threshold at $\epsilon = 5\times 10^{-2}$ and we observe that this is sufficient for our new algorithm to closely match the results of the existing algorithm. Finally, in all of the experiments here we use $2.5 \times 10^{-2}$ as the relative truncation threshold of the localized orbitals when counting the fraction of entries that are non-zero. This measure of ``locality'' has the advantage of not depending on the geometry of the physical system. However, for completeness we also provide spread computations in the final example. \rr{Prior work validates the expected exponential decay of the orbitals by varying the truncation threshold \cite{SCDM}.} 

All numerical results shown were run on a quad-socket
Intel Xeon E5-4640 processor clocked at 2.4 GHz with 1.5 TB of RAM and our algorithms were implemented in MATLAB R2015a. Our code is sequential and the only multi-threading is in the \mbox{LAPACK} and BLAS routines called by MATLAB. \rr{The storage cost of all the algorithms presented here is $\Or(n_e N),$ which is also the cost to store $\Psi$ in memory. In our largest example a copy of $\Psi$ cost roughly 60 GB to store.} The Kohn-Sham orbitals were computed using Quantum ESPRESSO \cite{QE} and VMD \cite{VMD} was used for plotting orbitals and molecules in the alkane and water examples.

\subsection{BH$_{3}$NH$_{3}$}
First, we demonstrate the performance of the approximate column
selection method for the dissociation process of a BH$_{3}$NH$_{3}$
molecule. The main purpose of this example is to demonstrate that our
approximate localization algorithm is equally effective when localized
orbitals form disconnected groups or a single group.  
Figure~\ref{fig:bh3nh3} shows the localized orbitals for three atomic configurations,
where the distance between B and N atoms is 1.18, 3.09, and 4.96 Bohr,
respectively and also shows the locality for orbitals
computed by Algorithms~\ref{alg:scdm}, \ref{alg:rand}, and
\ref{alg:refine} for each of the three atomic configurations. 
Here we plot an isosurface of the orbitals at a value of $2.5 \times 10^{-2}.$

We find
that Algorithm~\ref{alg:refine} automatically identifies that the
localized orbitals should be treated as two disconnected groups for the
dissociated configuration, and as one single group for the bonded
configuration.
We see that in all cases, the randomized method works
quite well on its own. Furthermore, after applying
Algorithm~\ref{alg:refine} to the output of the randomized method, the
sparsity of the orbitals is nearly indistinguishable from that of the original SCDM
algorithm. \rr{In these three scenarios the condition number of $P_{:,\CS},$ equivalently $\left(\Psi_{\CS,:}\right)^*,$ is never larger than three. Here we let $\CS \subset \CSt$ denote the final set of columns we have selected via the last QRCP in Algorithm~\ref{alg:refine} as it corresponds to the columns of the density matrix from which we ultimately build the localized basis.}

\begin{figure}[ht!]
  \centering
  \includegraphics[width = .6\textwidth]{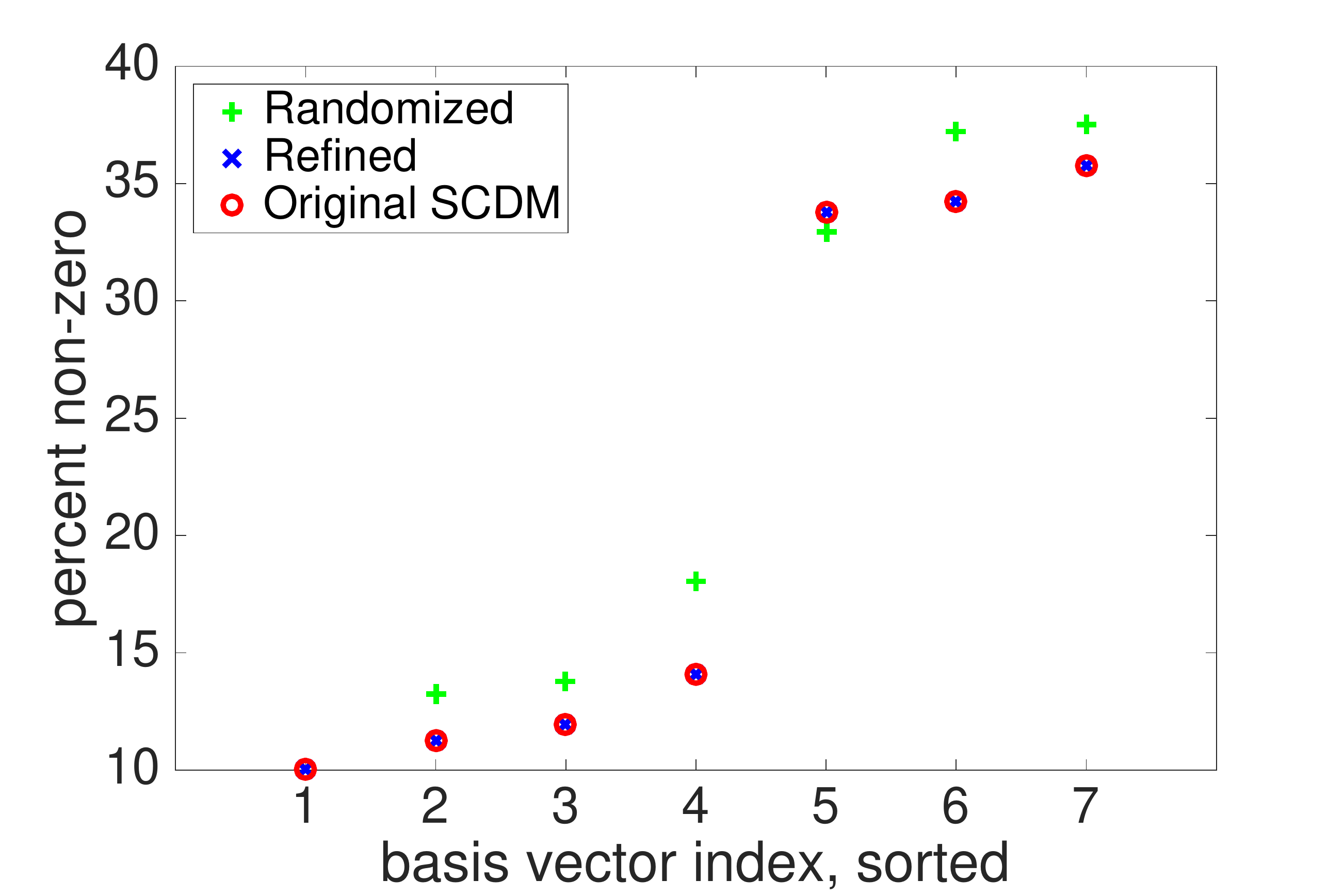}
  \includegraphics[width = .3\textwidth]{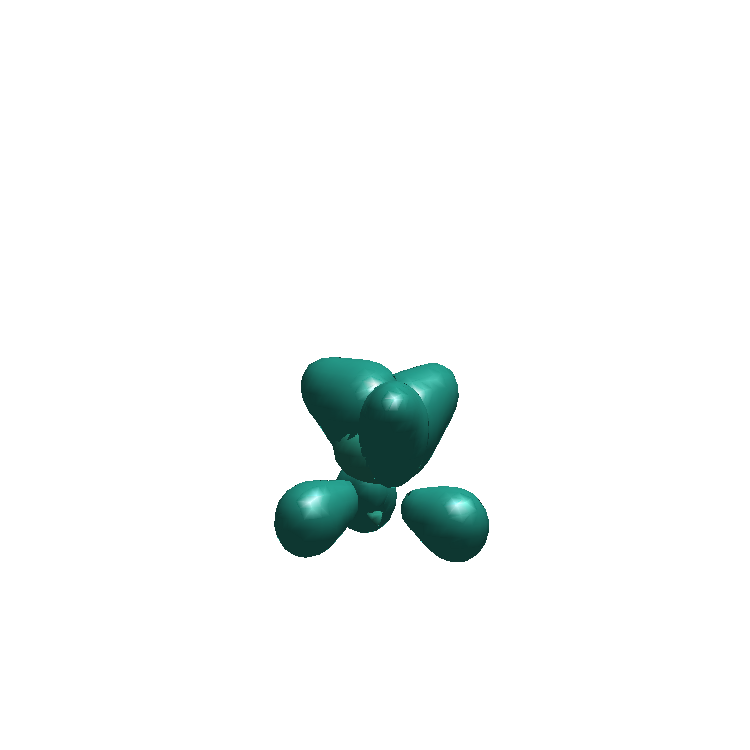}
  \includegraphics[width = .6\textwidth]{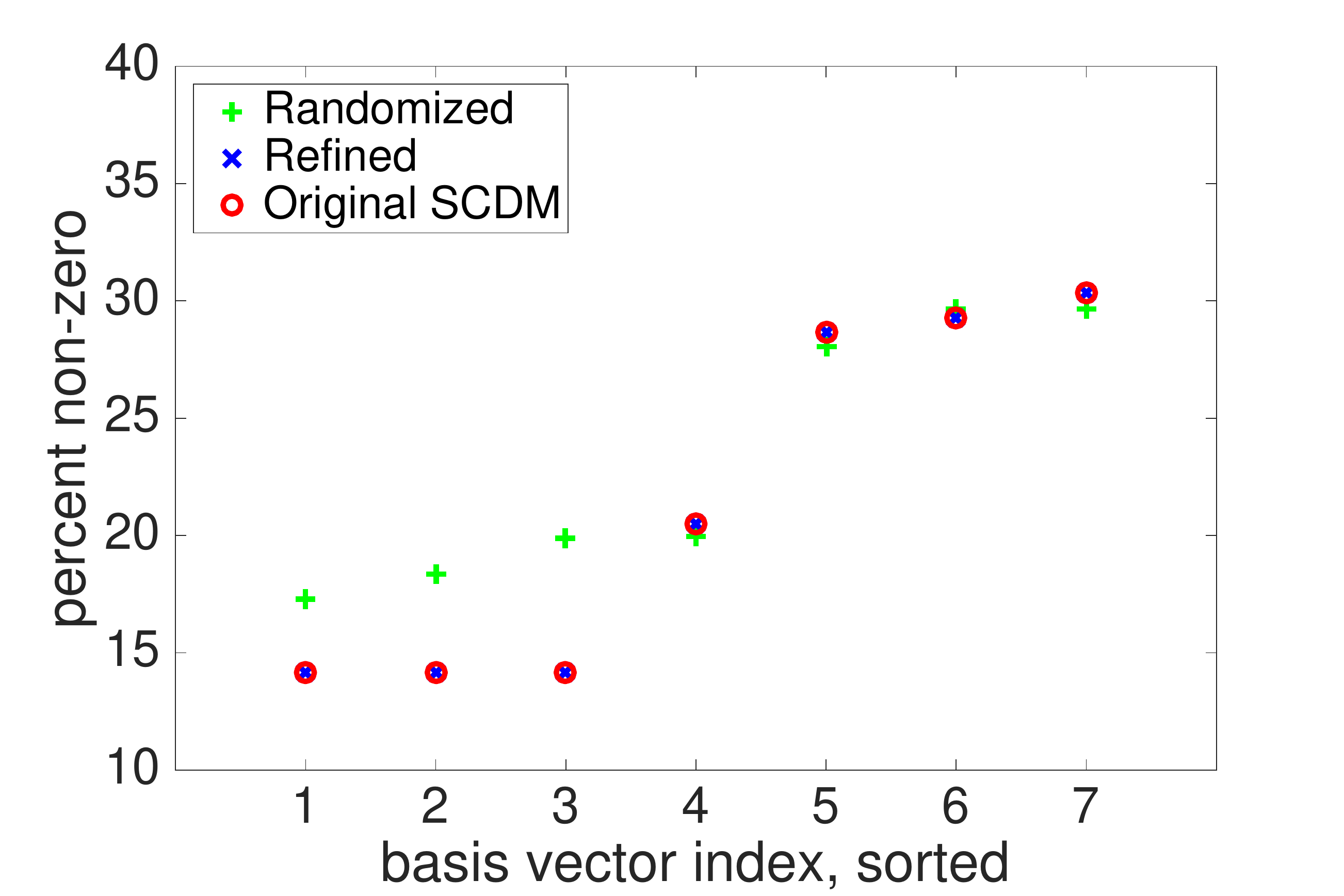}
  \includegraphics[width = .3\textwidth]{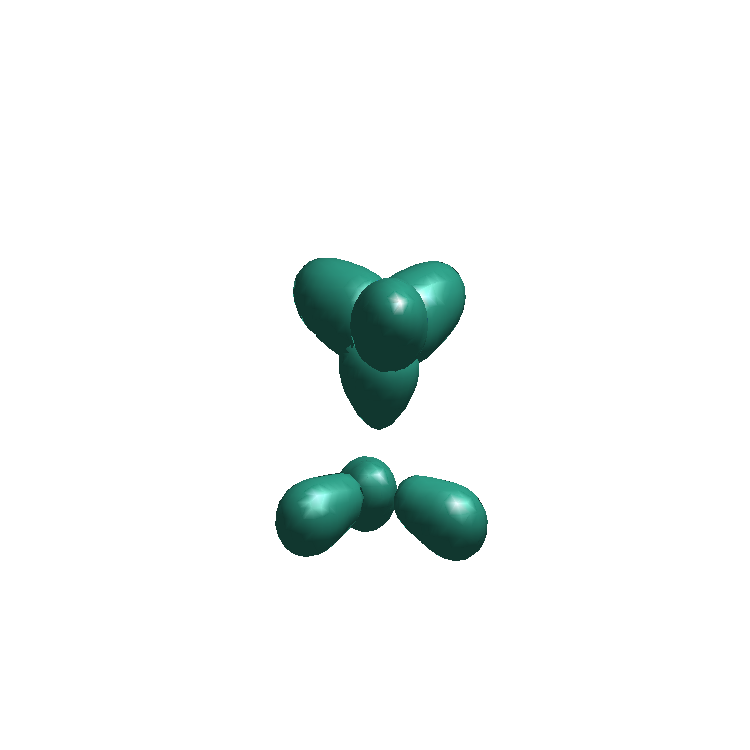}
  \includegraphics[width = .6\textwidth]{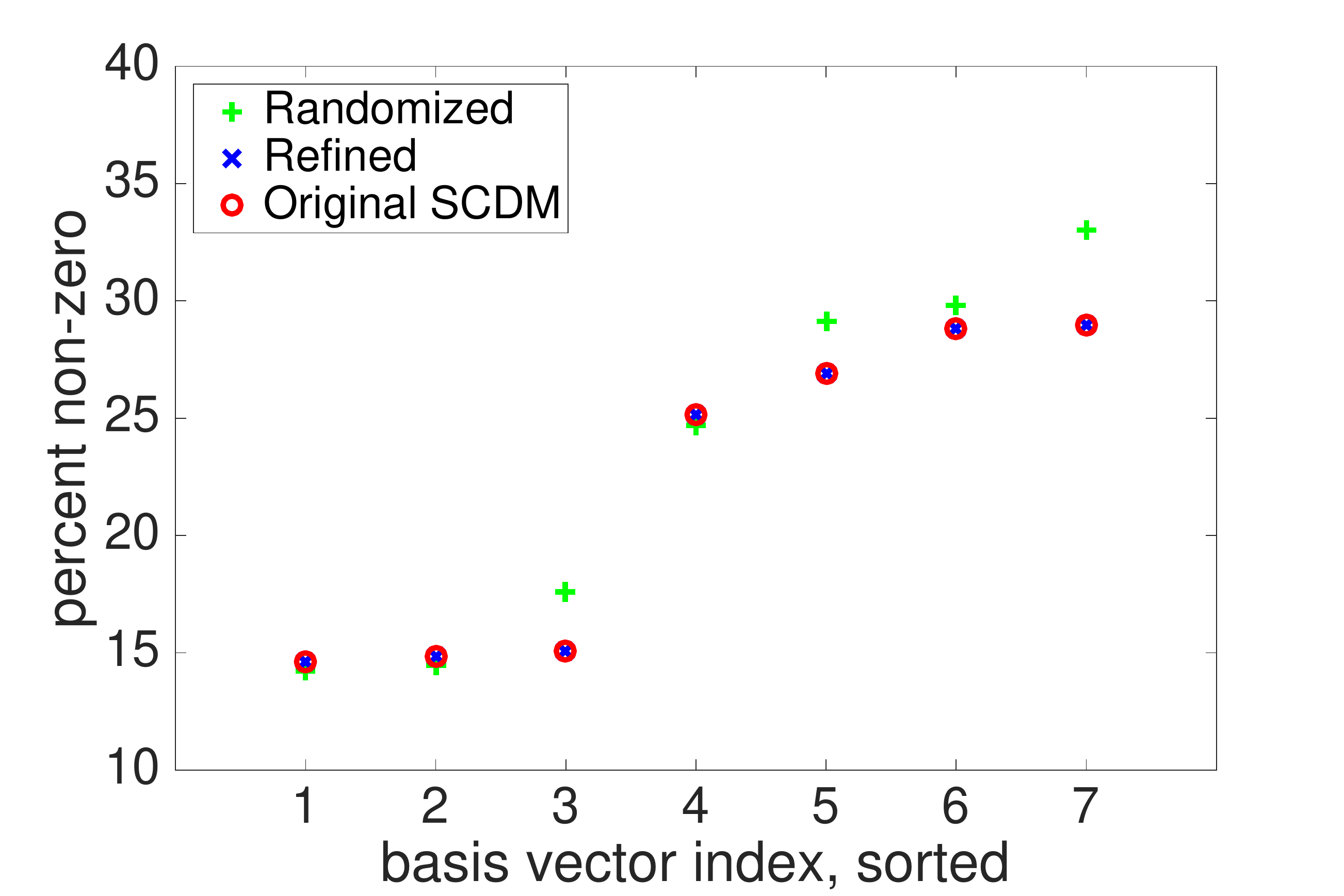}
  \includegraphics[width = .3\textwidth]{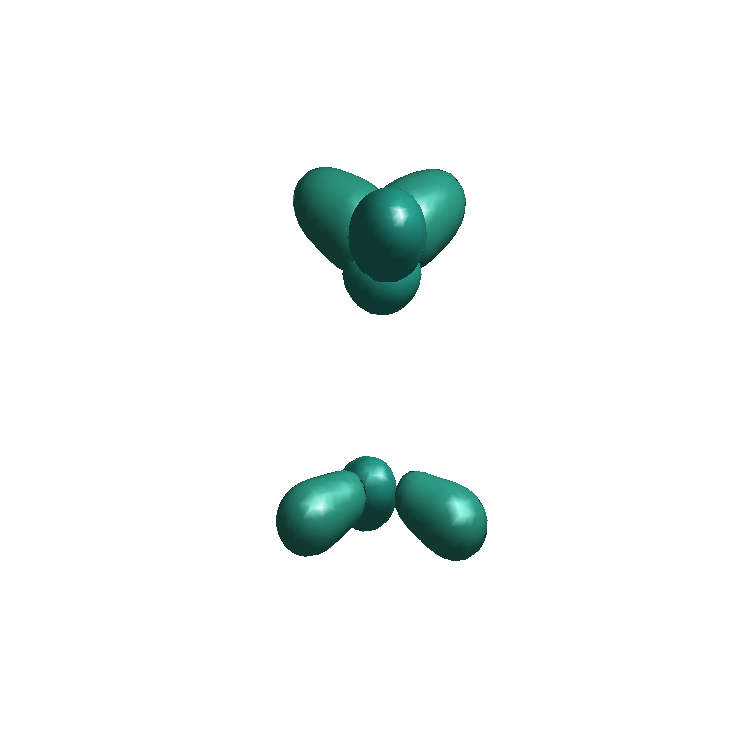}
  \caption{Sparsity (left) of localized orbitals computed by Algorithms~\ref{alg:scdm}, \ref{alg:rand}, and \ref{alg:refine} based on fraction of non-zero entries after
  truncation and orbital isosurfaces (right) at $2.5 \times 10^{-2}$ generated by Algorithm~\ref{alg:refine} when using the output of Algorithm~\ref{alg:rand} as input.
  Three different configurations moving from the bonded configuration
  (top) to the dissociated configuration (bottom).}
  \label{fig:bh3nh3}
\end{figure}

\subsection{Alkane chain}

Our second example is the alkane chain (atomic configuration shown in
Figure~\ref{fig:alkane}). Similar to the BH$_{3}$NH$_{3}$, this example
is not computationally expensive, but confirms that the approximate
localization algorithm is effective even when all localized orbitals
form one large connected group.
We demonstrate that the refinement process
still achieves the desired goal. In this example $N=820,125$ and $n_e =
100$.

\begin{figure}[ht!]
  \centering
  \includegraphics[width = .75\textwidth]{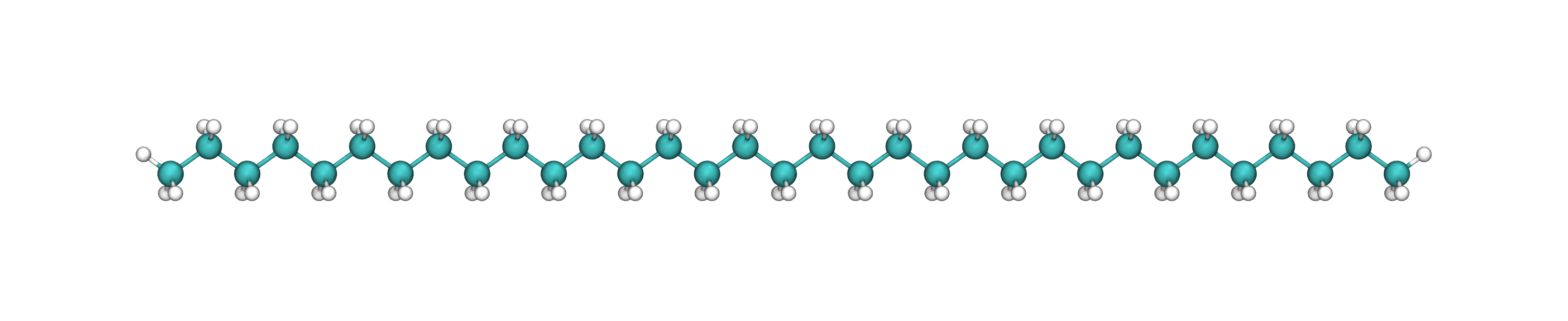}
  \caption{Atomic configuration of the alkane chain}
  \label{fig:alkane}
\end{figure}

Figure~\ref{fig:alkane_locality} shows histograms of the fraction of
non-zero entries after truncation for the randomized method, the
refinement procedure applied to the output of the randomized method, and
our original algorithm. We observe that the randomized method actually
serves to localize the orbitals rather well. However, the output is
clearly not as good as that produced by the original SCDM algorithm. However, once the
refinement algorithm is applied we see that, while not identical, the locality of the localized orbitals basically matches that of the localized orbitals generated by Algorithm~\ref{alg:scdm}. This is further
illustrated in Figure~\ref{fig:alkane_ex}, which shows isosurfaces for a localized orbital
generated by each of the three methods. \rr{Once again the columns of the density matrix we ultimately select are very well conditioned, in fact the condition number is less than two.}

\begin{figure}[ht!]
  \centering
  \includegraphics[width = 1\textwidth]{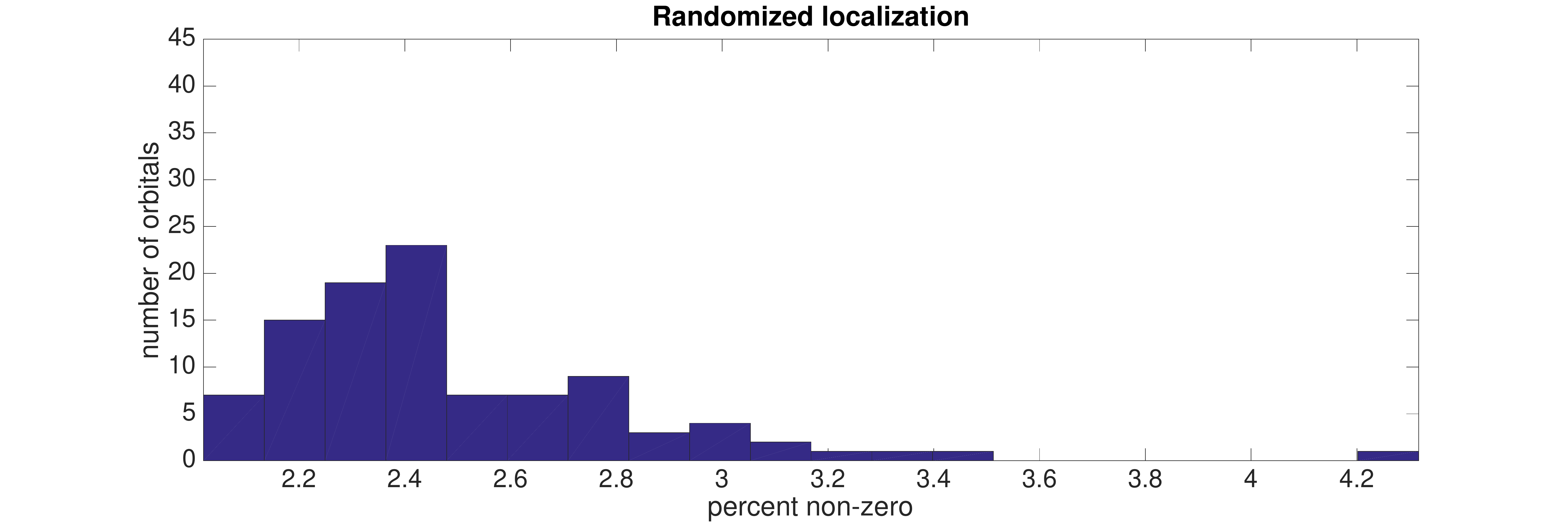}
  \includegraphics[width = 1\textwidth]{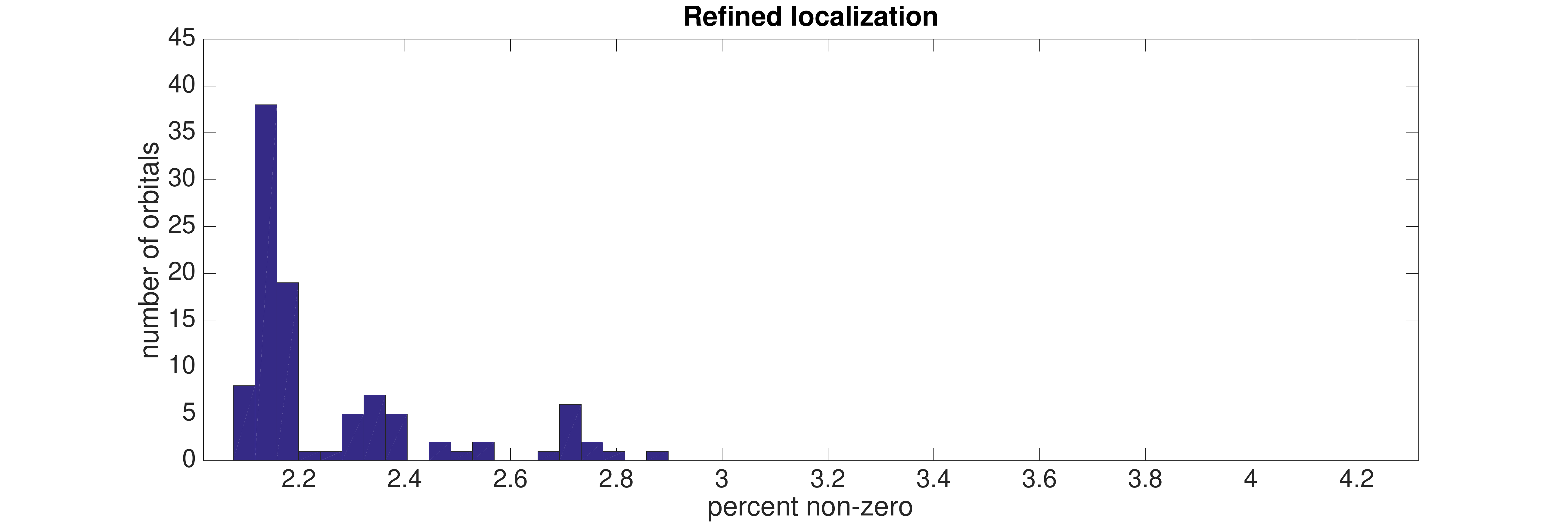}
  \includegraphics[width = 1\textwidth]{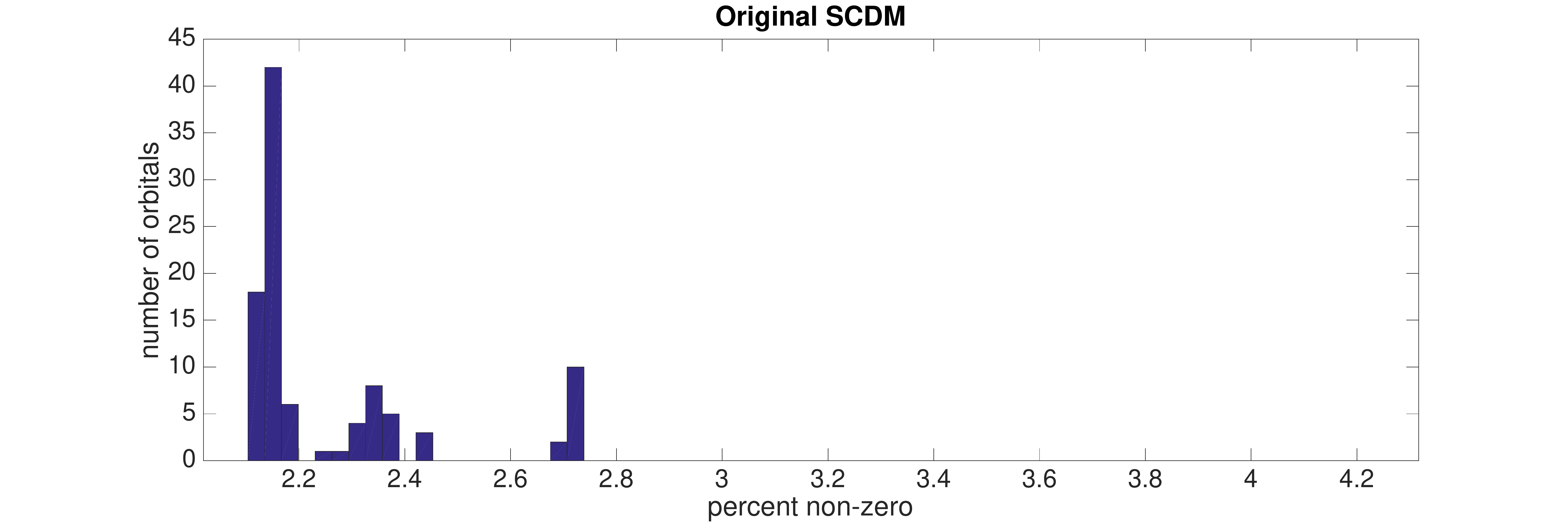}
  \caption{Histogram of localized orbitals for the alkane chain example computed by three different algorithms based on fraction of non-zero entries after truncation. (top) output of the randomized algorithm, (middle) output of the refinement algorithm applied to the output of the randomized algorithm, and (bottom) output of the original SCDM algorithm.}
  \label{fig:alkane_locality}
\end{figure}

\begin{figure}[ht!]
  \centering
  \includegraphics[width = .3\textwidth]{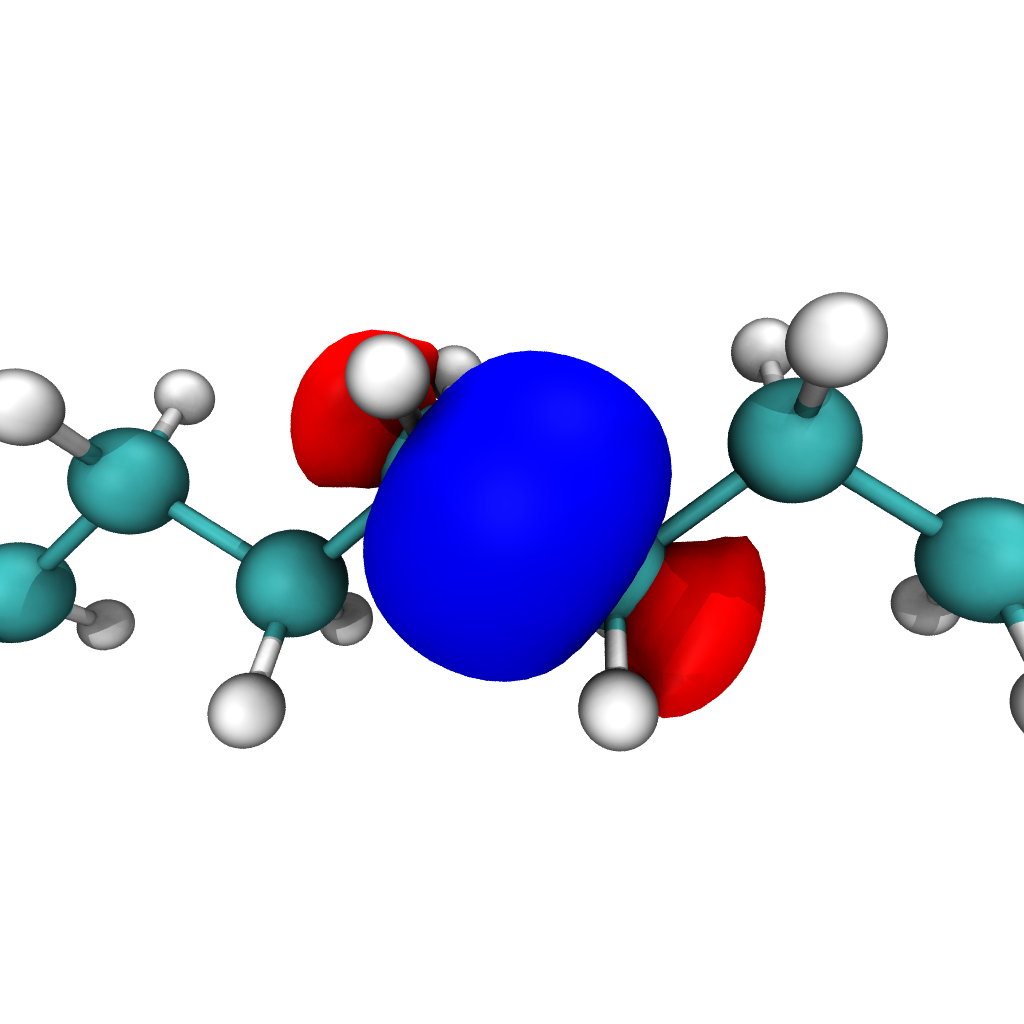}
  \includegraphics[width = .3\textwidth]{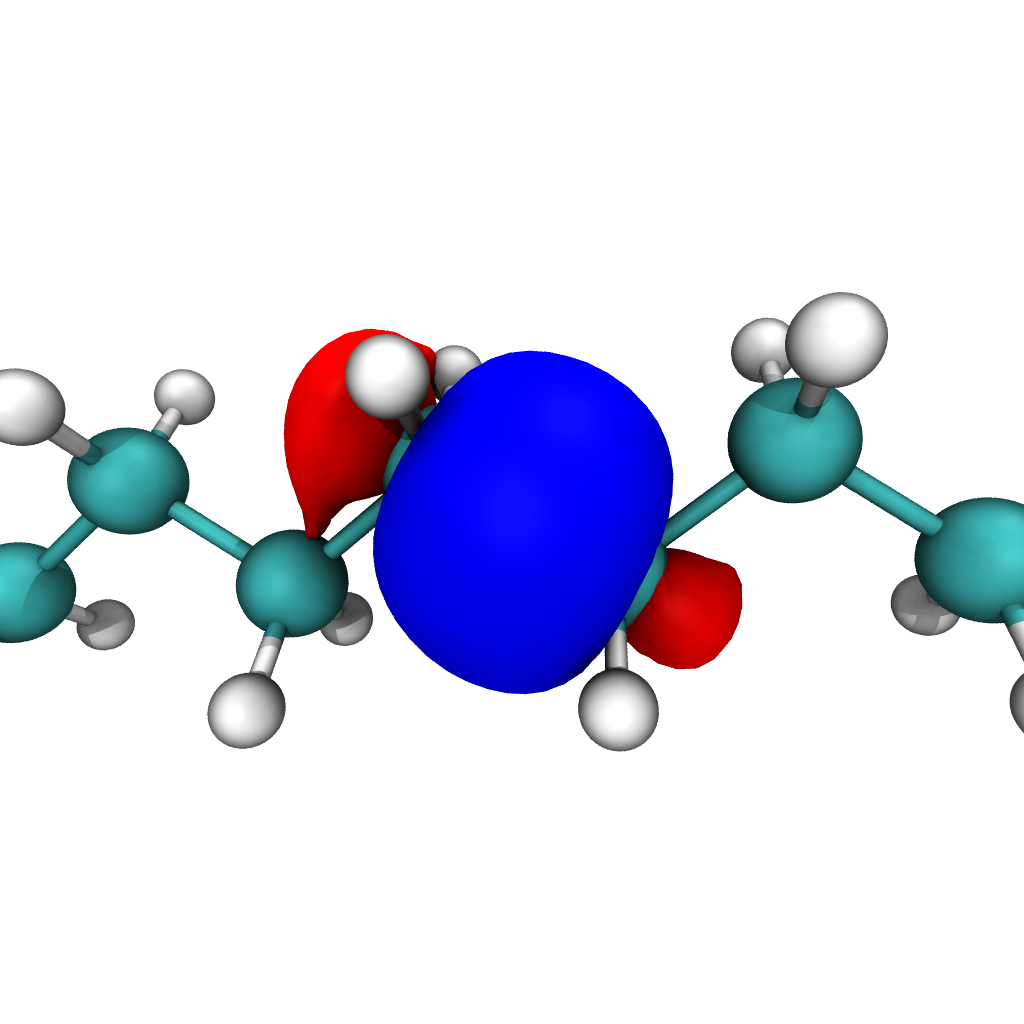}
  \includegraphics[width = .3\textwidth]{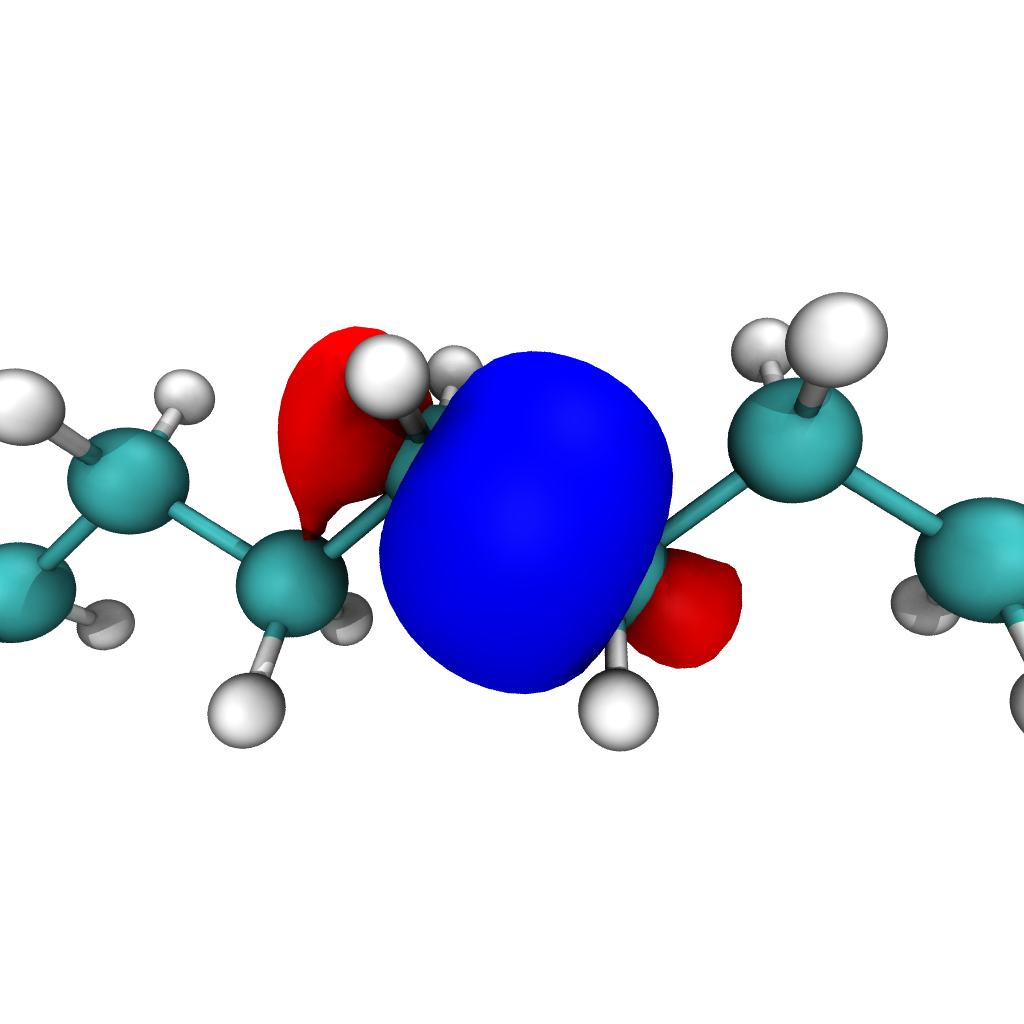}
  \caption{The same (as determined by picking those maximally correlated) orbital as generated by the randomized algorithm (left), the refinement procedure (middle), and the original SCDM algorithm (right). Here an isosurface value of $5\times 10^{-3}$ was used and the colors delineate positive and negative regions.}
  \label{fig:alkane_ex}
\end{figure}

Table~\ref{tab:alkane} illustrates the computational cost of our new
algorithms as compared to the original version. The refinement
algorithm is over nine times faster than the original algorithm, though
it cannot be perform without first approximately localizing the
orbitals.
Hence, we also provide the time of getting the orthogonal transform out
of Algorithm~\ref{alg:refine}, which is the sum of the preceding
three lines in the table. Even taking the whole pipeline into account we
see a speed up of about a factor of close to seven. 

\begin{table}
\caption{Runtime for localization algorithms as applied to an alkane chain.\label{tab:alkane}}
\centering
\begin{tabular}{|l|c|} \hline
 Operation & time (s) \\ \hhline{|=|=|}
 Matrix-matrix multiplication $\Psi Q$ & 0.2713 \\ \hline
 Randomized version, Algorithm \ref{alg:rand} & 0.0261 \\ \hline
 Refinement step, Algorithm \ref{alg:refine} & 0.7176 \\ \hline
 Total cost of our two stage algorithm & 1.0149 \\ \hline
 Original algorithm, Algorithm \ref{alg:scdm} & 6.9412 \\ \hline
\end{tabular}
\end{table}

\subsection{Water molecules}
We now consider a three dimensional system consisting of $256$ water molecules 
(part of the atomic configuration shown in Figure~\ref{fig:water}).
In this example, $N=7,381,125$ and $n_e = 1024$.

\begin{figure}[ht!]
  \centering
  \includegraphics[width = .4\textwidth]{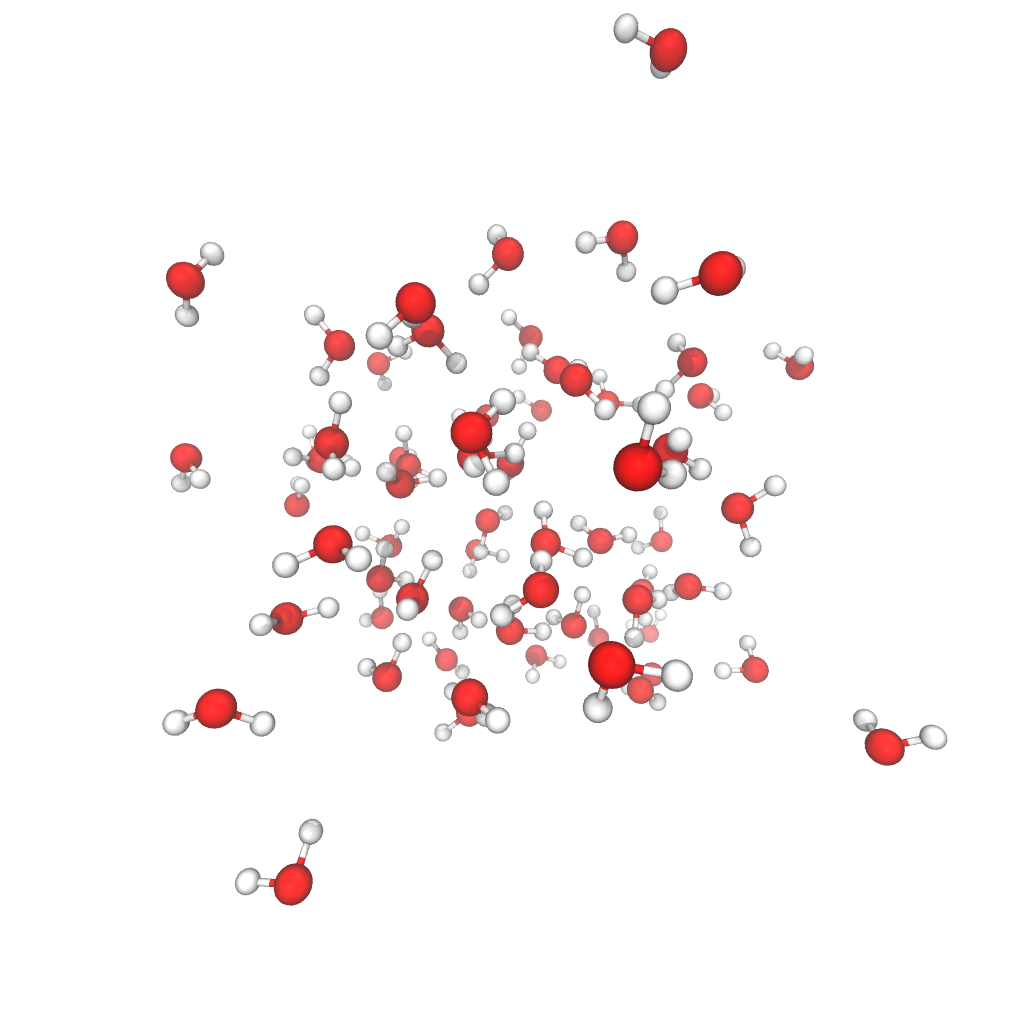}
  \caption{Part of the atomic configuration of 256 water molecules}
  \label{fig:water}
\end{figure}

Figure~\ref{fig:water_locality256} shows histograms of the fraction of
non-zero entries after truncation for the randomized method, the
refinement procedure applied to the output of the randomized method, and
the original SCDM algorithm when applied to the 256 water
molecule system. As before, the randomized method actually serves to localize the
orbitals rather well.
However, there is still visible difference between the result of the
randomized algorithm and the original SCDM algorithm. 
Application of the refinement
algorithm achieves a set of localized orbitals that broadly match the quality of the ones computed by the original SCDM algorithm. \rr{Similar to the previous two examples $P_{:,\CS}$ is very well conditioned, its condition number is once again less than two.}

\begin{figure}[ht!]
  \centering
  \includegraphics[width = 1\textwidth]{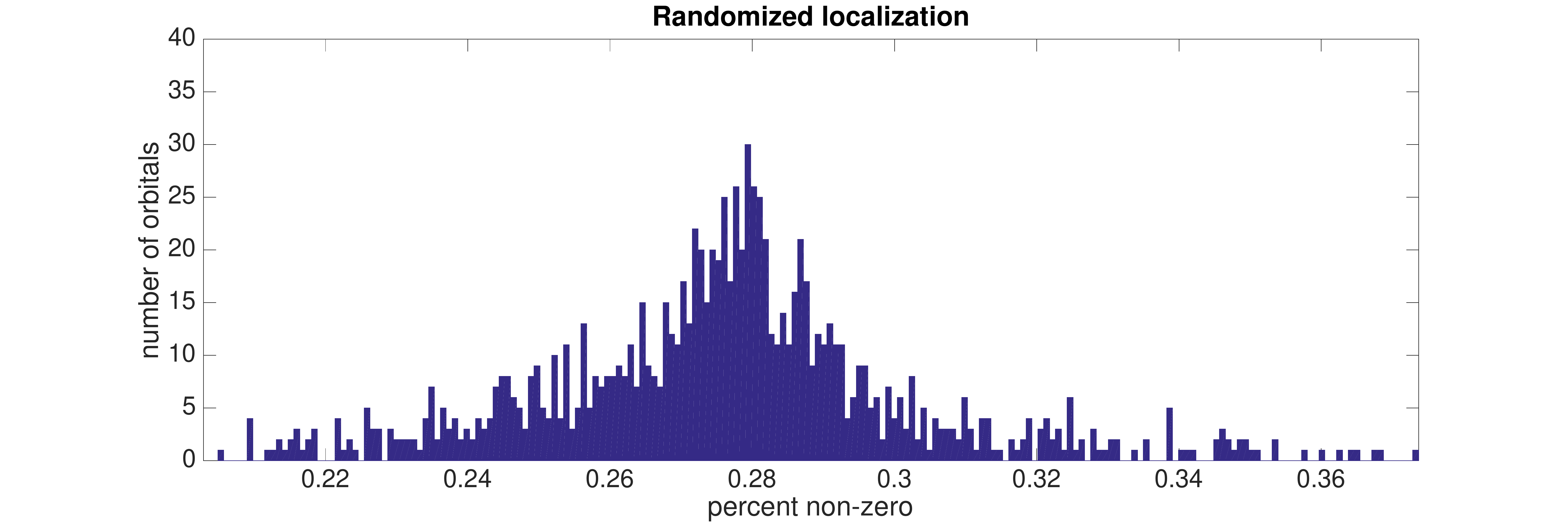}
  \includegraphics[width = 1\textwidth]{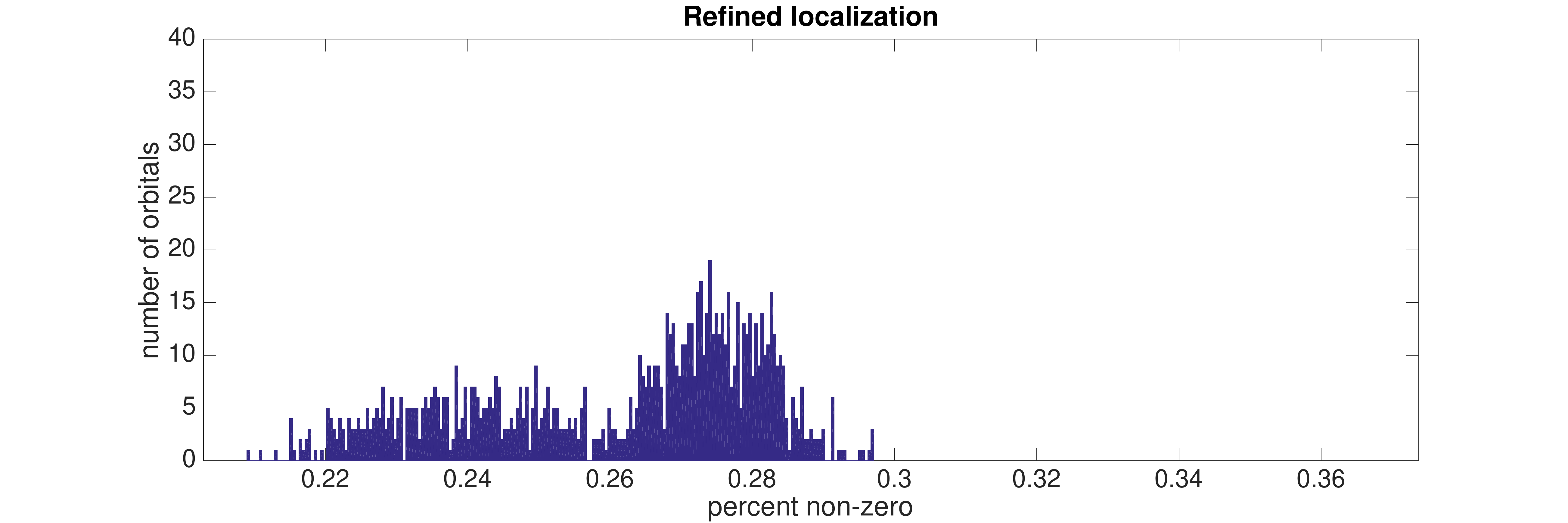}
  \includegraphics[width = 1\textwidth]{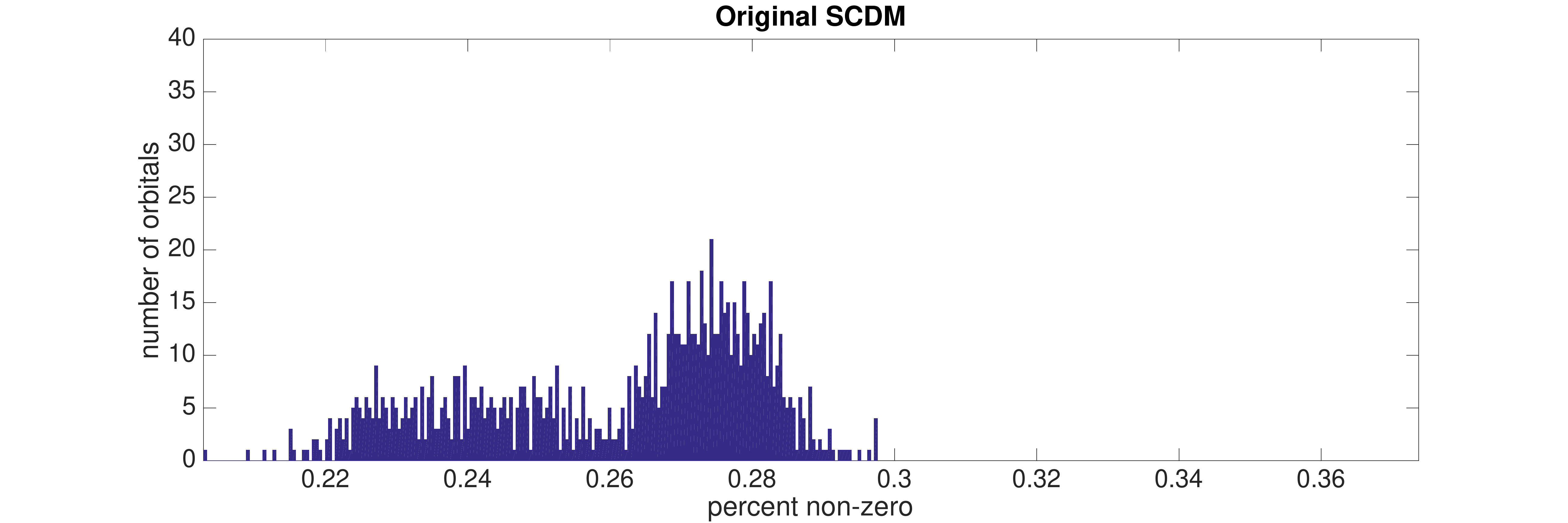}
  \caption{Histogram of localized orbitals for 256 water molecules computed by three different algorithms based on fraction of non-zero entries after truncation. (top) output of the randomized algorithm, (middle) output of the refinement algorithm applied to the output of the randomized algorithm, and (bottom) output of the original SCDM algorithm.}
  \label{fig:water_locality256}
\end{figure}

Table~\ref{tab:water256} illustrates the computational cost of our new
algorithms as compared to Algorithm~\ref{alg:scdm}. As
before, the randomized algorithm for computing is very fast and in
this case much faster than the matrix-matrix multiplication required to
construct the localized orbitals themselves. This makes the algorithm
particularly attractive in practice when $\rho$ is given and, as in many
electronic structure codes, the application of $Q$ to $\Psi$ may be
effectively parallelized. Furthermore, the complete procedure for getting orthogonal transform out of Algorithm~\ref{alg:refine}, the sum of the preceding three lines in the table, is more than 30 times faster than Algorithm~\ref{alg:scdm}.

\begin{table}
\caption{Runtime for localization algorithms as applied 256 water molecules.\label{tab:water256}}
\centering
\begin{tabular}{|l|c|c|} \hline
 Operation &  time (s) \\ \hhline{|=|=|}
 Matrix-matrix multiplication $\Psi Q$ & 47.831 \\ \hline
 Randomized version, Algorithm \ref{alg:rand} & 14.024 \\ \hline
 Refinement step, Algorithm \ref{alg:refine} & 78.361 \\ \hline
 Total cost of our two stage algorithm & 140.22 \\ \hline
 Original algorithm, Algorithm \ref{alg:scdm} & 4496.1 \\ \hline
\end{tabular}
\end{table}

Finally, we use this example to provide a comparison with the popular
Wannier90 \cite{wannier90} software package and further demonstrate the
quality of orbitals computed by our algorithms. While we have previously
been looking at sparsity after truncation to evaluate the quality of the
orbitals, we now also evaluate them based on the spread criteria Wannier90 tries to minimize. Loosely speaking this corresponds to the sum of
the variance of each orbital~\cite{MarzariVanderbilt1997}. All of the spreads here were computed by Wannier90 to ensure the same quantity was being measured in each case. Importantly, this is the quantity Wannier90 seeks to minimize and we therefore do not expect to do better under this metric. For example, given our localized orbitals as input, Wannier90 should always be able to at least slightly decrease the objective function value. For this comparison, we used the random initial guess option in Wannier90 and the default convergence tolerance of $10^{-10}.$ \rr{Notably, the random choice in Wannier90 does not correspond to a random gauge, rather an initial guess is constructed by randomly placing Gaussian functions in space.}

\begin{table}
\caption{Final spread of orbitals for 256 water molecules and time to
solution.}
\label{tab:wannier}
\centering
\begin{tabular}{|l|c|c|} \hline
 Algorithm & spread $\left(\text{\AA}^2\right)$ & time to solution (s) \\ \hhline{|=|=|=|}
 Algorithm~\ref{alg:scdm} & 589.91 & 4496.1 \\ \hline
 Randomized version, Algorithm \ref{alg:rand} & 636.60 & 14.024 \\ \hline
Two stage procedure using Algorithms~\ref{alg:rand} and~\ref{alg:refine} & 589.97 & 140.22 \\ \hline
 Wannier90 & 550.20 & 1715.2 \\ \hline
\end{tabular}
\end{table}

Table~\ref{tab:wannier} shows the spreads and time to solution for all of the algorithms. First, we observe that the spread using our two stage procedure almost exactly matches that of the output of Algorithm~\ref{alg:scdm} while being over an order of magnitude faster. Secondly, the output from our algorithms, even the randomized one, are close to the local minimum found by Wannier90. \rr{Figure~\ref{fig:wannier_spread} directly compares the spreads of our localized orbitals from the two stage algorithm and those at a local minimum found by Wannier90. As expected, there are discrepancies, but we are generally within approximately $0.05 \text{\AA}^2$ per orbital. This aligns with our expectations based on the objective function difference. For comparison, we computed a localized basis for an isolated $\text{H}_2\text{O}$ molecule. Using Algorithm~\ref{alg:scdm} yielded spreads of 0.46, 0.48, 0.56, and 0.57 $\text{\AA}^2$. Subsequently running Wannier90 to convergence yielded spreads of 0.47, 0.48, 0.55, and 0.55 $\text{\AA}^2$, which matched the results when starting with a Wannier90 generated random initial guess.} 

\begin{figure}[ht!]
  \centering
  \includegraphics[width = 1\textwidth]{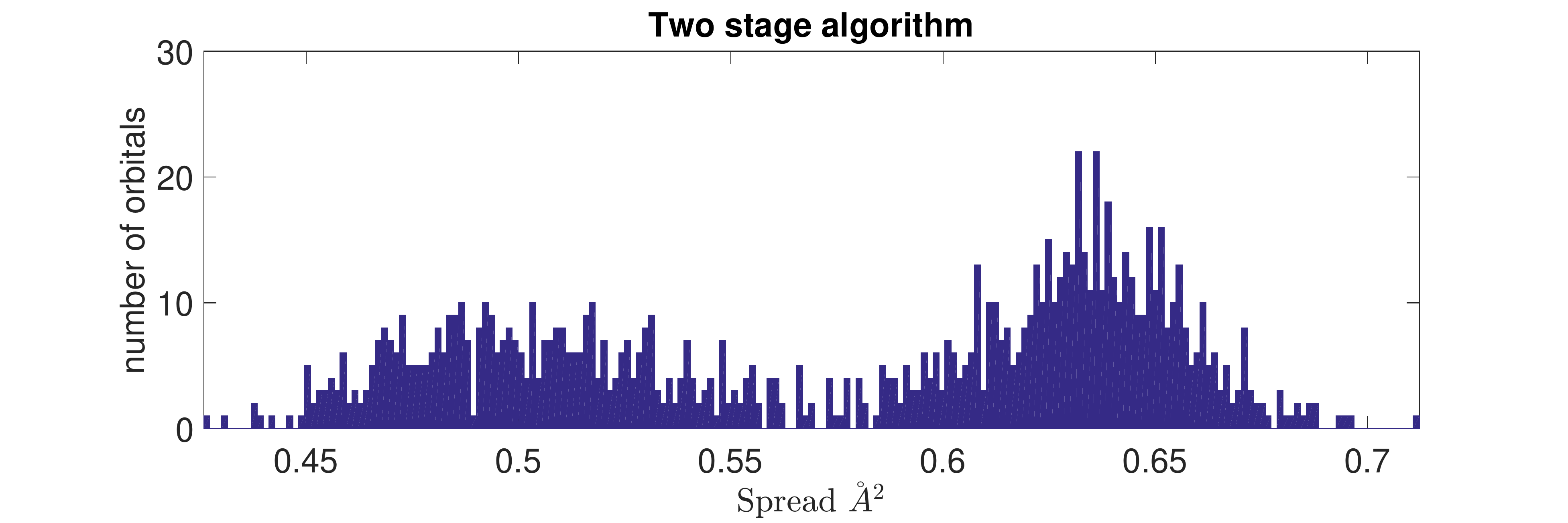}
  \includegraphics[width = 1\textwidth]{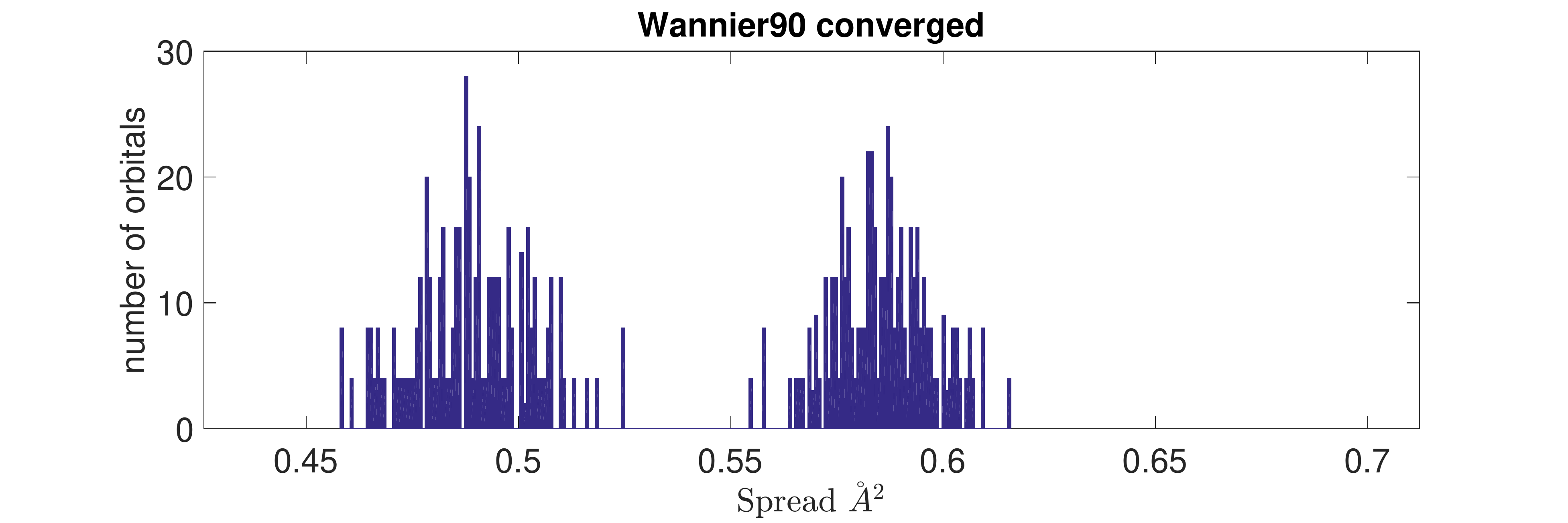}
  \caption{\rr{Histogram of localized orbitals for 256 water molecules computed by two different algorithms based on the Wannier90 spread functional. (top) output of the refinement algorithm applied to the output of the randomized algorithm, and (bottom) converged local minimum for Wannier90.}}
  \label{fig:wannier_spread}
\end{figure}

Admittedly, the time to solution for Wannier90 may depend strongly on the initial guess, \egc~in this experiment the initial spread for Wannier90 was $5339 \text{ \AA}^2$ and convergence took 20 iterations. However, a poor initial guess could result in Wannier90 failing to converge, converging to a worse local minimum, \rr{or taking longer to converge}. Our algorithms are direct and have no such dependence on an initial guess. In this experiment, each iteration of Wannier90 took roughly 85 seconds. So even two iterations costs more than our two stage algorithm and we have omitted the cost of computing the overlap matrices based on $\Psi$, which Wannier90 requires as input and we computed using Quantum ESPRESSO on 256 processors. Collectively, these results make our algorithm an attractive alternative to Wannier90. If a local minimum of the objective is desired, the output from \rr{Algorithms~\ref{alg:rand} or~\ref{alg:refine}} may be used as an algorithmically constructed initial guess for Wannier90.

\begin{remark}
\rr{Seeding Wannier90 with our two stage algorithm convergence took 16 iterations and we appear to arrive at an equivalent local minima. For this system, it would appear the area around the local minimum is quite flat and the default absolute convergence criteria is rather stringent. In fact, after one iteration starting with our algorithm Wannier90 yields localized functions whose spread is withing $0.4\%$ of the converged value. To get to within $1\%$ starting from the random initial guess takes five iterations.} 
\end{remark}

\section{Conclusion}
\label{sec:conclusion}

We have presented a two stage algorithm to accelerate the
computation of the SCDM for finding localized representation of the Kohn-Sham invariant
subspace. 
We first utilize an
algorithm based on random sampling to approximately localize the basis
and then perform a subsequent refinement step. This method can achieve
computational gains of over an order of magnitude for systems of
relatively large sizes.
Furthermore, the orbitals computed are qualitatively and quantitatively similar to those generated by the
original SCDM algorithm. Lastly, for large systems we observe that our algorithm may provide an attractive alternative to Wannier90 and, at the very least, may provide a simple method for the construction of a good intial guess.

Rapid computation of a localized basis allows for its use within various
computations in electronic structure calculation where a localized basis
may need to be computed repeatedly. This includes computations such as
molecular dynamics and time dependent density functional theory with hybrid exchange-correlation functions.
Finally, the
ideas inherent to the SCDM procedure have potentially applicability in
problems outside of Kohn-Sham DFT because the structural and behavioral
properties we exploit are not necessarily unique to the problem. Our new
algorithms would admit faster computation in such contexts as well.

\section*{Acknowledgments}
A.D. is partially supported by a National Science Foundation Mathematical Sciences Postdoctoral Research Fellowship under grant number DMS-1606277. L.L. is supported by the DOE Scientific
Discovery through the Advanced Computing (SciDAC) program, the DOE Center for Applied Mathematics for Energy Research Applications (CAMERA) program, and the Alfred P.  Sloan fellowship. L.Y. is supported by the National Science Foundation under award DMS-1328230 and DMS-1521830 and the U.S. Department of Energy’s Advanced Scientific Computing Research program under award DE-FC02-13ER26134/DE-SC0009409.
The authors thank Stanford University and the Stanford Research Computing Center for providing computational resources and support that have contributed to these research results. The authors also thank the anonymous referees for their many helpful suggestions. 

\bibliographystyle{siam}
\bibliography{fastQR}

\begin{thebibliography}{10}

\bibitem{lapack}
{\sc E.~Anderson, Z.~Bai, C.~Bischof, S.~Blackford, J.~Demmel, J.~Dongarra,
  J.~Du~Croz, A.~Greenbaum, S.~Hammarling, A.~McKenney, and D.~Sorensen}, {\em
  {LAPACK} Users' Guide}, SIAM, Philadelphia, PA, third~ed., 1999.

\bibitem{AquilantePedersenMerasEtAl2006}
{\sc F.~Aquilante, T.~B. Pedersen, A.~S. de~Mer{\'a}s, and H.~Koch}, {\em Fast
  noniterative orbital localization for large molecules}, J. Chem. Phys., 125
  (2006), p.~174101.

\bibitem{Becke1993}
{\sc A.~D. Becke}, {\em Density functional thermochemistry. iii. the role of
  exact exchange}, J. Chem. Phys., 98 (1993), pp.~5648--5652.

\bibitem{BenziBoitoRazouk2013}
{\sc M.~Benzi, P.~Boito, and N.~Razouk}, {\em Decay properties of spectral
  projectors with applications to electronic structure}, SIAM Rev., 55 (2013),
  pp.~3--64.

\bibitem{Scalapack}
{\sc L.~S. Blackford, J.~Choi, A.~Cleary, E.~D'Azevedo, J.~Demmel, I.~Dhillon,
  J.~Dongarra, S.~Hammarling, G.~Henry, A.~Petitet, K.~Stanley, D.~Walker, and
  R.~C. Whaley}, {\em {ScaLAPACK} Users' Guide}, Society for Industrial and
  Applied Mathematics, Philadelphia, PA, 1997.

\bibitem{Blount}
{\sc E.I. Blount}, {\em Formalisms of band theory}, vol.~13 of Solid State
  Phys., Academic Press, 1962, pp.~305--373.

\bibitem{Brouder2007}
{\sc C.~Brouder, G.~Panati, M.~Calandra, C.~Mourougane, and N.~Marzari}, {\em
  Exponential localization of wannier functions in insulators}, Phys. Rev.
  Lett., 98 (2007), p.~046402.

\bibitem{Cloizeaux1964b}
{\sc J.~Cloizeaux}, {\em Analytical properties of $n$-dimensional energy bands
  and {W}annier functions}, Phys. Rev., 135 (1964), pp.~A698--A707.

\bibitem{Cloizeaux1964a}
\leavevmode\vrule height 2pt depth -1.6pt width 23pt, {\em Energy bands and
  projection operators in a crystal: Analytic and asymptotic properties}, Phys.
  Rev., 135 (1964), pp.~A685--A697.

\bibitem{SCDM}
{\sc A.~Damle, L.~Lin, and L.~Ying}, {\em Compressed representation of
  {K}ohn--{S}ham orbitals via selected columns of the density matrix}, J. Chem.
  Theory Comput., 11 (2015), pp.~1463--1469.

\bibitem{SCDMk}
\leavevmode\vrule height 2pt depth -1.6pt width 23pt, {\em Scdm-k: Localized
  orbitals for solids via selected columns of the density matrix}, Journal of
  Computational Physics, 334 (2017), pp.~1 -- 15.

\bibitem{DemmelRRQR}
{\sc J.~Demmel, L.~Grigori, M.~Gu, and H.~Xiang}, {\em Communication avoiding
  rank revealing qr factorization with column pivoting}, Tech. Report
  UCB/EECS-2013-46, EECS Department, University of California, Berkeley, May
  2013.

\bibitem{Gu}
{\sc J.~A. {Duersch} and M.~{Gu}}, {\em {True BLAS-3 Performance QRCP using
  Random Sampling}}, ArXiv e-prints,  (2015).

\bibitem{ELiLu2010}
{\sc W.~E, T.~Li, and J.~Lu}, {\em Localized bases of eigensubspaces and
  operator compression}, Proc. Natl. Acad. Sci., 107 (2010), pp.~1273--1278.

\bibitem{FosterBoys1960}
{\sc J.~M. Foster and S.~F. Boys}, {\em Canonical configurational interaction
  procedure}, Rev. Mod. Phys., 32 (1960), pp.~300--302.

\bibitem{QE}
{\sc P.~Giannozzi, S.~Baroni, N.~Bonini, M.~Calandra, R.~Car, C.~Cavazzoni,
  D.~Ceresoli, G.~L Chiarotti, M.~Cococcioni, I.~Dabo, A.~{Dal Corso},
  S.~de~Gironcoli, S.~Fabris, G.~Fratesi, R.~Gebauer, U.~Gerstmann,
  C.~Gougoussis, A.~Kokalj, M.~Lazzeri, L.~Martin-Samos, N.~Marzari, F.~Mauri,
  R.~Mazzarello, S.~Paolini, A.~Pasquarello, L.~Paulatto, C.~Sbraccia,
  S.~Scandolo, G.~Sclauzero, A.~P Seitsonen, A.~Smogunov, P.~Umari, and R.~M
  Wentzcovitch}, {\em Quantum espresso: a modular and open-source software
  project for quantum simulations of materials}, J. Phys.: Condens. Matter, 21
  (2009), p.~395502.

\bibitem{GVL}
{\sc G.~H. Golub and C.~F. Van~Loan}, {\em Matrix computations}, Johns Hopkins
  Univ. Press, Baltimore, third~ed., 1996.

\bibitem{Gygi2009}
{\sc F.~Gygi}, {\em Compact representations of {K}ohn--{S}ham invariant
  subspaces}, Phys. Rev. Lett., 102 (2009), p.~166406.

\bibitem{GygiDuchemin2012}
{\sc F.~Gygi and I.~Duchemin}, {\em Efficient computation of {H}artree--{F}ock
  exchange using recursive subspace bisection}, J. Chem. Theory Comput., 9
  (2012), pp.~582--587.

\bibitem{hasan2010colloquium}
{\sc M~Zahid Hasan and Charles~L Kane}, {\em Colloquium: topological
  insulators}, Reviews of Modern Physics, 82 (2010), p.~3045.

\bibitem{HohenbergKohn1964}
{\sc P.~Hohenberg and W.~Kohn}, {\em {Inhomogeneous electron gas}}, Phys. Rev.,
  136 (1964), pp.~B864--B871.

\bibitem{VMD}
{\sc W.~Humphrey, A.~Dalke, and K.~Schulten}, {\em {VMD} -- {V}isual
  {M}olecular {D}ynamics}, Journal of Molecular Graphics, 14 (1996),
  pp.~33--38.

\bibitem{Kohn1996}
{\sc W.~Kohn}, {\em Density functional and density matrix method scaling
  linearly with the number of atoms}, Phys. Rev. Lett., 76 (1996),
  pp.~3168--3171.

\bibitem{KohnSham1965}
{\sc W.~Kohn and L.~Sham}, {\em {Self-consistent equations including exchange
  and correlation effects}}, Phys. Rev., 140 (1965), pp.~A1133--A1138.

\bibitem{LinLu2015}
{\sc L.~{Lin} and J.~{Lu}}, {\em {Sharp decay estimates of discretized Green's
  functions for Schr$\backslash$''odinger type operators}}, arXiv:1511.07957,
  (2015).

\bibitem{MahoneyDrineas}
{\sc M.~W. Mahoney and P.~Drineas}, {\em Cur matrix decompositions for improved
  data analysis}, Proceedings of the National Academy of Sciences, 106 (2009),
  pp.~697--702.

\bibitem{Martin2004}
{\sc R.~Martin}, {\em Electronic Structure -- Basic Theory and Practical
  Methods}, Cambridge Univ. Pr., West Nyack, {NY}, 2004.

\bibitem{Martinsson}
{\sc P.~G. {Martinsson}}, {\em {Blocked rank-revealing QR factorizations: How
  randomized sampling can be used to avoid single-vector pivoting}}, ArXiv
  e-prints,  (2015).

\bibitem{WannierReview}
{\sc N.~Marzari, A.~A. Mostofi, J.~R. Yates, I.~Souza, and D.~Vanderbilt}, {\em
  Maximally localized {W}annier functions: Theory and applications}, Rev. Mod.
  Phys., 84 (2012), pp.~1419--1475.

\bibitem{MarzariVanderbilt1997}
{\sc N.~Marzari and D.~Vanderbilt}, {\em Maximally localized generalized
  {W}annier functions for composite energy bands}, Phys. Rev. B, 56 (1997),
  pp.~12847--12865.

\bibitem{wannier90}
{\sc A.~A. Mostofi, J.~R. Yates, Y.~Lee, I.~Souza, D.~Vanderbilt, and
  N.~Marzari}, {\em wannier90: A tool for obtaining maximally-localised wannier
  functions}, Comp. Phys. Comm., 178 (2008), pp.~685 -- 699.

\bibitem{MustafaCohCohenEtAl2015}
{\sc J.~I. Mustafa, S.~Coh, M.~L. Cohen, and S.~G. Louie}, {\em {Automated
  construction of maximally localized Wannier functions: Optimized projection
  functions method}}, Phys. Rev. B, 92 (2015), p.~165134.

\bibitem{Nenciu}
{\sc G.~Nenciu}, {\em Existence of the exponentially localised {W}annier
  functions}, Comm. Math. Phys., 91 (1983), pp.~81--85.

\bibitem{OzolinsLaiCaflischEtAl2013}
{\sc V.~Ozoli{\c{n}}{\v{s}}, R.~Lai, R.~Caflisch, and S.~Osher}, {\em
  Compressed modes for variational problems in mathematics and physics}, Proc.
  Natl. Acad. Sci., 110 (2013), pp.~18368--18373.

\bibitem{PerdewErnzerhofBurke1996}
{\sc J.~P. Perdew, M.~Ernzerhof, and K.~Burke}, {\em {Rationale for mixing
  exact exchange with density functional approximations}}, J. Chem. Phys., 105
  (1996), pp.~9982--9985.

\bibitem{ProdanKohn2005}
{\sc E.~Prodan and W.~Kohn}, {\em {Nearsightedness of electronic matter}},
  Proc. Natl. Acad. Sci., 102 (2005), pp.~11635--11638.

\bibitem{WuSelloniCar2009}
{\sc X.~Wu, A.~Selloni, and R.~Car}, {\em Order-{N} implementation of exact
  exchange in extended insulating systems}, Phys. Rev. B, 79 (2009), p.~085102.

\end{thebibliography}
\end{document}